\documentclass[a4paper]{article}

\usepackage{amsthm}
\usepackage{amsmath}
\usepackage{amssymb}
\usepackage{asymptote}
\usepackage{bm}
\usepackage{booktabs}
\usepackage{enumitem}
\usepackage[hang,flushmargin]{footmisc}
\usepackage[left=25mm,right=25mm,top=25mm,bottom=25mm]{geometry}
\usepackage[
	pdftitle={},
	pdfauthor={Norman Do, Jian He, and Daniel V. Mathews},
	ocgcolorlinks,
	linkcolor=linkblue,==
	citecolor=linkred,
	urlcolor=linkred]
{hyperref}
\usepackage{mathpazo}
\usepackage{mathtools}
\usepackage{microtype}
\usepackage{multicol}
\usepackage{sectsty}
\usepackage{setspace}
\usepackage{titlesec}
	\titlespacing{\section}{0pt}{12pt}{0pt}
	\titlespacing{\subsection}{0pt}{6pt}{0pt}
\usepackage[nottoc]{tocbibind}

\theoremstyle{plain}
\newtheorem{theorem}{Theorem}
\newtheorem{proposition}[theorem]{Proposition}

\newtheorem{lemma}[theorem]{Lemma}

\newtheorem*{conjecture*}{Conjecture}

\theoremstyle{definition}
\newtheorem{definition}[theorem]{Definition}

\theoremstyle{remark}

\definecolor{linkred}{rgb}{0.75,0,0}
\definecolor{linkblue}{rgb}{0,0,1}

\newcommand\blfootnote[1]{
  \begingroup
  \renewcommand\thefootnote{}\footnote{#1}
\addtocounter{footnote}{-1}
  \endgroup
}

\sectionfont{\large}
\subsectionfont{\normalsize}

\setlist{nolistsep}

\newcommand{\Z}{\mathbb{Z}}

\newcommand{\mmu}{\boldsymbol{\mu}}

\makeatletter
\newcommand{\dashrule}[1][black]{
  \color{#1}\rule[\dimexpr.5ex-.2pt]{4pt}{.4pt}\xleaders\hbox{\rule{4pt}{0pt}\rule[\dimexpr.5ex-.2pt]{4pt}{.4pt}}\hfill\kern0pt
}
\newcommand{\rulecolor}[1]{
  \def\CT@arc@{\color{#1}}
}
\makeatother

\begin{document}

{\large \bfseries Counting non-crossing permutations on surfaces of any genus}

{\bfseries Norman Do, Jian He, and Daniel V. Mathews}

\emph{Given a surface with boundary and some points on its boundary, a polygon diagram is a way to connect those points as vertices of non-overlapping polygons on the surface. Such polygon diagrams represent non-crossing permutations on a surface with any genus and number of boundary components. If only bigons are allowed, then it becomes an arc diagram. The count of arc diagrams is known to have a rich structure. We show that the count of polygon diagrams exhibits the same interesting behaviours, in particular it is almost polynomial in the number of points on the boundary components, and the leading coefficients of those polynomials are the intersection numbers on the compactified moduli space of curves $\overline{\mathcal{M}}_{g,n}$.} 

\blfootnote{ 
\emph{Date:} \today \\ The first author was supported by Australian Research Council grant  DP180103891. The third author was supported by Australian Research Council grant DP160103085.}

~

\hrule

\setlength{\parskip}{0pt}
\tableofcontents
\setlength{\parskip}{6pt}

~

\hrule

\section{Introduction} \label{sec:intro}

A {\em polygon} on a connected compact oriented surface $S$ with boundary is an embedded (closed) disc bounded by a sequence of properly embedded arcs $P_1 P_2, P_2 P_3, \ldots, P_{m-1} P_m, P_m P_1$, where $P_1, P_2, \ldots, P_m \in \partial S$. The points $P_1, \ldots, P_m$ are called the \emph{vertices} of the polygon and the arcs $P_i P_{i+1}$ (with $i$ taken mod $m$) are its \emph{edges}. Given a finite set of \emph{marked points} $M \subset \partial S$, a {\em polygon diagram} on $(S,M)$ is a disjoint union of polygons on $S$ whose vertices are precisely the marked points $M$. See figure \ref{polygon-diagram} for an example. Two polygon diagrams $D_1, D_2$ on $(S,M)$ are {\em equivalent} if there is an orientation preserving homeomorphism $\phi: S\to S$ such that $\phi|_{\partial S}$ is the identity and $\phi(D_1) = D_2$. 

Polygon diagrams are closely related to non-crossing permutations. In this paper we count them.

\begin{figure}
\begin{center}
\includegraphics[scale=0.65]{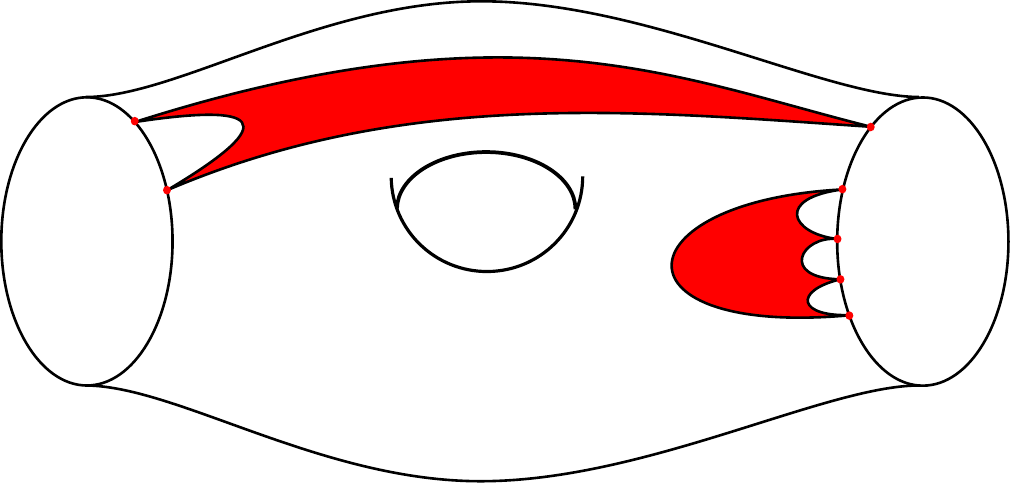}
\caption{A polygon diagram on $S_{1,2}$.}
\label{polygon-diagram}
\end{center}
\end{figure}

Denote by $S_{g,n}$ a connected compact oriented surface of genus $g$ with $n \geq 1$ boundary components, or just $S$ when $g$ and $n$ are understood. 
Label the boundary components of $S$ as $F_1, \ldots, F_n$. Since we will be performing cutting and pasting operations on polygon diagrams, it is often helpful to choose a single vertex $\mathbf{m}_i \in M \cap F_i$ to be a decorated marked point on each boundary component $F_i$ containing at least one vertex (i.e. such that $M \cap F_i \neq \emptyset$). Two polygon diagrams $D_1, D_2$ on $S$ can then be regarded as equivalent 
if there is an orientation preserving diffeomorphism of $S$ taking $D_1$ to $D_2$, such that each decorated marked point on $D_1$ is mapped to the decorated marked point of $D_2$ on the same boundary component.
Fixing the total number of vertices on each boundary component $F_i$ to be $\mu_i$ (i.e. $|M \cap F_i| = \mu_i$), let $P_{g,n}(\mu_1, \ldots, \mu_n)$ be number of equivalence classes of polygon diagrams on $(S,M)$. Clearly
$P_{g,n}$ only depends on $g,n,\mu_1, \ldots, \mu_n$ (not on the choice of particular $S$ or $M$) and is a symmetric function of the variables $\mu_1, \ldots, \mu_n$. 
\begin{proposition}\label{Pexamples}
\begin{align}
\label{eqn:P01_formula}
P_{0,1}(\mu_1)&=
\begin{cases}
\binom{2\mu_1-1}{\mu_1}\frac{2}{\mu_1+1}, & \mu_1 > 0 \\
1, & \mu_1 = 0
\end{cases}\\
\label{eqn:P02_formula}
P_{0,2}(\mu_1,\mu_2)&=
\begin{cases} \binom{2\mu_1-1}{\mu_1}\binom{2\mu_2-1}{\mu_2}\left(\frac{2\mu_1\mu_2}{\mu_1+\mu_2}+1\right), & \mu_1,\mu_2 > 0 \\ 
\binom{2\mu_1-1}{\mu_1}, & \mu_2=0 
\end{cases} \\
P_{0,3}(\mu_1,\mu_2,\mu_3)&= \binom{2\mu_1-1}{\mu_1}\binom{2\mu_2-1}{\mu_2}\binom{2\mu_3-1}{\mu_3}\left(2\mu_1\mu_2\mu_3+\sum_{i\neq j}\mu_i\mu_j+\sum^3_{i=1} \frac{\mu_i^2-\mu_i}{2\mu_i-1}+1\right) \\
P_{1,1}(\mu_1)&=\binom{2\mu-1}{\mu} \frac{1}{2\mu-1} \frac{\mu^3 + 3\mu^2 + 20\mu - 12}{12} \label{P11}
\end{align}
\end{proposition}
Here we take the convention $\binom{-1}{0}=1$ when $\mu_i$ is $0$.

Suppose $D$ is a polygon diagram on $(S,M)$ where $S$ is a disc or an annulus, i.e. $(g,n) = (0,1)$ or $(0,2)$. Each boundary component $F_i$ inherits an orientation from $S$. Label the marked points of $M$ by the numbers $1,2,\ldots, |M| = \sum_{i=1}^n \mu_i$, in order around $F_1$ in the disc case, and in order around $F_1$ then $F_2$ in the annulus case. 
Orienting each polygon in agreement with $S$ induces a cyclic order on the vertices (and vertex labels) of each polygon, giving the cycles of a permutation $\pi$ of $\{1,\ldots \sum \mu_i\}$. Such a permutation is known as a
{\em non-crossing permutation} if $S$ is a disc, or {\em annular non-crossing permutation} if $S$ is an annulus. We say the diagram $D$ \emph{induces} or \emph{represents} the permutation $\pi$. 

Non-crossing permutations are well known combinatorial objects. It is a classical result 
that the number of non-crossing permutations on the disc is a Catalan number. Annular non-crossing permutations were (so far as we know) first 
introduced by King \cite{King1999}. They were studied in detail by Mingo--Nica \cite{MN2004}, Nica--Oancea \cite{NO2009}, Goulden--Nica--Oancea \cite{GNO2011}, 
Kim \cite{Kim2012} and Kim--Seo--Shin \cite{KSS2014}. 

In general, if we number the marked points $M$ from $1$ to $|M| = \sum_{i=1}^n \mu_i$ in order around the oriented boundaries $F_1$, then $F_2$, up to $F_n$, then in a similar way, a polygon diagram represents a non-crossing permutation on a surface with arbitrary genus and an arbitrary number of boundary components. This paper studies such non-crossing permutations via polygon diagrams.

The relation between permutation and genus here differs slightly from others in the literature. The notion of genus of a permutation $\pi$ in \cite{Jacques1968} and subsequent papers such as \cite{CH2013, CH2014, CH2018}, in our language, is the \emph{smallest} genus $g$ of a surface $S$ with \emph{one boundary component} on which a polygon diagram exists representing the permutation $\pi$; equivalently, it is the genus of a surface $S$ with one boundary component on which a polygon diagram exists representing $\pi$, such that all the components of $S \backslash D$ are discs. This differs again from the notion of genus of a permutation in \cite{CFHMMPS2002}.

Given a non-crossing permutation $\pi$ on the disc, it's clear that there is a unique polygon diagram $D$ (up to equivalence) representing $\pi$. Therefore $P_{0,1}(\mu)$ is also the $\mu$-th Catalan number. Uniqueness of representation is also true for \emph{connected} annular non-crossing partitions. An annular non-crossing partition is \emph{connected} if there is at least one edge between the two boundary components, i.e. from $F_1$ to $F_2$. Uniqueness of representation follows since an edge from $F_1$ to $F_2$ cuts the annulus into a disc. 
The number of connected annular non-crossing partitions counted in $P_{0,2}(\mu_1, \mu_2)$ is known to be \cite[cor. 6.8]{MN2004}
\begin{align*}
\binom{2\mu_1-1}{\mu_1}\binom{2\mu_2-1}{\mu_2}\left(\frac{2\mu_1\mu_2}{\mu_1+\mu_2}\right),
\end{align*}
which appears as a term in the formula (\ref{eqn:P02_formula}) for $P_{0,2}(\mu_1,\mu_2)$. A disconnected annular non-crossing permutation however 
can be represented by several distinct polygon diagrams, and $P_{0,2}$ can be viewed as the total count of annular non-crossing permutations with
multiplicities. Similarly, in general the $P_{g,n}(\mu, \ldots, \mu_n)$ can be regarded as counts with multiplicity of non-crossing permutations on arbitrary connected compact oriented surfaces with boundary.
 
If all polygons in $D$ are bigons, then collapsing them into arcs turns $D$ into an {\em arc diagram} previously studied by the first and third authors with Koyama 
\cite{DKM2017}. The count of arc diagrams exhibits quasi-polynomial behaviour, and the asymptotic behaviour is governed by intersection
numbers on the moduli space of curves. In this paper we show that the count of polygon diagrams has the same structure.
The arguments mirror those in \cite{DKM2017}. 

The formulae for $P_{g,n}$ in Proposition \ref{Pexamples} suggest that $P_{g,n}(\mu_1, \ldots, \mu_n)$ is a product of the $\binom{2\mu_i-1}{\mu_i}$, together with a rational function of the $\mu_i$'s. In fact we also know the form of the denominator. Moreover, the behaviour is \emph{better} than for arc diagrams in the sense that we obtain \emph{polynomials} rather than quasi-polynomials. 

\begin{theorem}\label{Pcount}
For $(g,n)\neq (0,1),(0,2)$, let $a=3g-3+n \geq 0$, and
\[
C_{g,n}(\mu)= 
\frac{1}{(2\mu-1)(2\mu-3)\dots(2\mu -2a -1)}\binom{2\mu-1}{\mu}
\]
Then 
\[
P_{g,n} (\mu_1, \ldots, \mu_n) = \left(\prod_{i=1}^{n}{C_{g,n}(\mu_i)}\right)F_{g,n}(\mu_1, \ldots, \mu_n)
\]
where $F_{g,n}$ is a polynomial with rational coefficients.
\end{theorem}
Note that $F_{g,n}$ might have some common factors with $(2\mu_i-1)(2\mu_i-3)\dots(2\mu_i -2a -1)$, which would
simplify the formula for $P_{g,n}$. For example, $F_{1,1}$ has a factor
$(2\mu_1-3)$, hence only $(2\mu_1-1)$ appears on the denominator in \eqref{P11}.

The $P_{g,n}$ satisfy a recursion which allows the count on a surface to be computed from the counts on surfaces with simpler topology, i.e, either smaller genus $g$, or fewer boundary components $n$, or fewer vertices $\mu_i$.

Let $X= \{1, 2, 3, \ldots, n\}$. For each $I\subseteq X$, let $\mmu_I = \{\mu_i \; \mid \; i\in I\}$.
\begin{theorem}
\label{thm:P_recursion}
For non-negative integers $g$ and $\mu_1, \ldots, \mu_n$ such that $\mu_1 > 0$, we have
\begin{align}
P_{g,n} (\mu_1, \ldots, \mu_n)
&= P_{g,n} (\mu_1 - 1, \mmu_{X\setminus \{1\}}) + \sum_{k=2}^n \mu_k P_{g,n-1} ( \mu_1 + \mu_k - 1, \mmu_{X\setminus \{1,k\}} ) \nonumber \\
& \quad + \mathop{\sum_{i+j=\mu_1 - 1}}_{j>0} \bigg[ P_{g-1,n+1} (i,j, \mmu_{X\setminus \{1\}}) + \mathop{\sum_{g_1 + g_2 = g}}_{I \sqcup J = X\setminus \{1\}} P_{g_1, |I|+1} (i, \mmu_I) \, P_{g_2, |J|+1} (j, \mmu_j) \bigg]. \label{precur-eq}
\end{align}
\end{theorem}

An edge $P_1P_2$ is {\em boundary parallel} if it cuts off a disc from the surface $S$. It is easy to create polygons using edges that are parallel to the same boundary component. The counts of these polygons are clearly combinatorial in nature instead of reflecting the underlying topology of $S$. Therefore from a topological point of view, it is natural to count polygon diagrams where none of the edges are boundary parallel. We call such a diagram a {\em pruned polygon diagram}. Let the count of pruned polygon diagrams be $Q_{g,n}(\mu_1, \ldots, \mu_n)$, i.e. the number of equivalence classes of pruned polygon diagrams on a surface of genus $g$, with $n$ boundary components, containing $\mu_1, \ldots, \mu_n$ marked points respectively. Clearly $Q_{g,n}(\mu_1, \ldots, \mu_n)$ is also a symmetric function of $\mu_1, \ldots, \mu_n$. As the name suggests, the relationship between $P_{g,n}$ and $Q_{g,n}$ mirrors that of Hurwitz numbers and pruned Hurwitz numbers \cite{DN2013}. It also mirrors the relationship between the counts of arc diagrams $G_{g,n}$ and non boundary-parallel arc diagrams $N_{g,n}$ in \cite{DKM2017}; we call the latter \emph{pruned arc diagrams}.

We call a function $f(\mu_1, \ldots, \mu_n)$ a \emph{quasi-polynomial} if it is given by a family of polynomial functions, depending on whether each of the integers $\mu_1, \ldots, \mu_n$ is zero, odd, or even (and nonzero).
In other words, a quasi-polynomial can be viewed as a collection of $3^n$ polynomials, 
depending on whether each $\mu_i$ is zero, odd, or nonzero even. 
Our definition of a quasi-polynomial differs slightly from the standard definition, in that $0$ is treated as a separate case rather than an even number.
More precisely, for each partition $X = X_e \sqcup X_o \sqcup X_\emptyset$, there is
a single polynomial $f^{(X_e,X_o,X_\emptyset)}( \mmu_{X_e}, \mmu_{X_o})$ such that $f(\mu_1, \ldots, \mu_n)=f^{(X_e,X_o,X_\emptyset)}(\mmu_{X_e}, \mmu_{X_o} )$ whenever $\mu_i = 0$ for $i\in X_\emptyset$, $\mu_i$ is nonzero and even for $i \in X_e$, and $\mu_i$ is odd for $i \in X_o$. (Here as above, for a set $I \subseteq X$, $\mmu_I = \{ \mu_i \; \mid \; i \in I \}$.)  A quasi-polynomial is \emph{odd} if
each $f^{(X_e,X_o,X_\emptyset)}(\mmu_{X_e}, \mmu_{X_o})$ is an odd polynomial with respect to each $\mu_i\in X_e\sqcup X_o$.

\begin{theorem}
\label{thm:quasipolynomiality}
For $(g,n) \neq (0,1)$ or $(0,2)$, $Q_{g,n}(\mu_1, \ldots, \mu_n)$ is an odd quasi-polynomial. 
\end{theorem}

The pruned diagram count captures topological information of $S_{g,n}$. The highest 
degree coefficients of the quasi-polynomial $Q_{g,n}$ are determined by intersection numbers in the compactified moduli 
space $\overline{\mathcal{M}}_{g,n}$.
\begin{theorem}\label{intersection}
For $(g,n) \neq (0,1)$ or $(0,2)$, $Q_{g,n}^{(X_e,X_o,X_\emptyset)}(\mu_1, \ldots, \mu_n)$ has degree $6g-6+3n$.
The coefficient $c_{d_1,\ldots,d_n}$ of the highest degree monomial $\mu_1^{2d_1+1}\cdots\mu_n^{2d_n+1}$
is independent of the partition $(X_e,X_o)$, and
$$c_{d_1,\ldots,d_n} = \frac{1}{2^{g-1}d_1!\cdots d_n!}\int_{\overline{\mathcal{M}}_{g,n}}\psi_1^{d_1}\cdots\psi_n^{d_n}.$$
\end{theorem}
Here $\psi_i$ is the Chern class of the $i$-th tautological line bundle over the compactified moduli space $\overline{\mathcal{M}}_{g,n}$ of genus $g$ curves with $n$ marked points.

\section{Preliminaries}
\label{sec:preliminaries}

In this section we state some identities required in the sequel. 

\subsection{Combinatorial identities}
\label{sec:comb_id}

The combinatorial identities required involve sums of binomial coefficients, multiplied by polynomials. The sums have a polynomial structure, analogous to the sums in \cite[defn. 5.5]{DKM2017} and \cite{NS2014}.

\begin{proposition}
\label{almostpoly}
For any integer $\alpha \geq 0$ there are polynomials $P_{\alpha}$ and $Q_{\alpha}$ such that
\begin{align*}
\sum_{0\leq i\leq n \text{ even}}{i^{2\alpha+1}\binom{2n}{n-i}} &= \frac{\binom{2n}{n}}{(2n-1)(2n-3)\dots(2n-2\alpha-1)}P_{\alpha}(n) \\
\sum_{0\leq i\leq n \text{ odd}}{i^{2\alpha+1}\binom{2n}{n-i}} &= \frac{\binom{2n}{n}}{(2n-1)(2n-3)\dots(2n-2\alpha-1)}Q_{\alpha}(n).
\end{align*}
\end{proposition}

In particular, when $\alpha = 0, 1$ we have 
\begin{equation}
\label{PQ01}
P_0(n) = \frac{1}{2}(n^2 - n), \quad
Q_0 (n) = \frac{1}{2} n^2, \quad
P_1 (n) = \left( n^2 - 1 \right)^2 n^2
\quad \text{and} \quad
Q_1 (n) = \frac{1}{2} n^2 \left( 2n^2 - 4n + 1 \right).
\end{equation}
In other words, we have identities
\begin{gather}
\label{comb_id_oe1}
\sum_{0\leq \nu\leq n \text{ even}}{\nu\binom{2\mu}{\mu-\nu}} = \frac{\binom{2\mu}{\mu}}{2\mu-1}\frac{\mu^2-\mu}{2}, \quad \quad \quad
\sum_{0\leq \nu\leq n \text{ odd}}{\nu\binom{2\mu}{\mu-\nu}} = \frac{\binom{2\mu}{\mu}}{2\mu-1}\frac{\mu^2}{2} \\
\label{comb_id_e3}
\sum_{0\leq \nu\leq n \text{ even}}{\nu^3\binom{2\mu}{\mu-\nu}} = \frac{\binom{2\mu}{\mu}}{(2\mu-1)(2\mu-3)}(\mu^2-1)^2\mu^2 \\
\label{comb_id_3o}
\sum_{0\leq \nu\leq n \text{ odd}}{\nu^3\binom{2\mu}{\mu-\nu}} = \frac{\binom{2\mu}{\mu}}{(2\mu-1)(2\mu-3)}\frac{\mu^2(2\mu^2-4\mu+1)}{2}
\end{gather}

\subsection{Algebraic results and identities}

We also need some results for summing polynomials over integers satisfying constraints on their sum and parities. 
They can be proved as in \cite{DKM2017} using generalisations of Ehrhart's theorem as in \cite{BV1997}, but we give more elementary proofs in the appendix.

\begin{proposition}\label{lemma-odd}
For positive odd integers $k_1$, $k_2$
$$\sum_{\substack{i_1,i_2\geq 1, \ \ i_1+i_2=n \\ \{i_1,i_2\} \text{ {have fixed parities}}}}i_1^{k_1}i_2^{k_2}$$ 
is an odd polynomial of degree $(k_1+k_2+1)$ in $n$. Furthermore the leading coefficient is independent of the choice of parities.
\end{proposition}

In other words, in the sum above, we fix elements $\varepsilon_1, \varepsilon_2 \in \Z/2\Z$ and the sum is over integers $i_1, i_2$ such that $i_1, i_2 \geq 1$, $i_1 + i_2 = n$ and $i_1 \equiv \varepsilon_1$ mod $2$, $i_2 \equiv \varepsilon_2$ mod $2$.

Proposition \ref{lemma-odd} can be directly generalized by induction to  the following.
\begin{proposition}\label{lemma-odd-induction}
For positive odd integers $k_1, k_2,\ldots ,k_m$
$$\sum_{\substack{i_1,i_2,\ldots,i_m\geq 1, \ \ i_1+i_2+\ldots +i_m=n \\ \{i_1,i_2,\ldots,i_m\} \text{ have fixed parities}}}i_1^{k_1}i_2^{k_2}\cdots i_m^{k_m}$$ 
is an odd polynomial of degree $(\sum_{i=1}^m k_i + m -1)$ in $n$. Furthermore the leading coefficient is independent of the choice of parities. \qed
\end{proposition}

We will need the following particular cases, which can be proved by a straightforward induction, and follow immediately from the discussion in the appendix.
\begin{lemma}\label{lem-odd-even-power-sums}
Let $n \geq 0$ be an integer.
\begin{enumerate}
\item
When $n$ is odd,
$\displaystyle
\sum_{\substack{0 \leq i \leq n \\ i \text{ odd}}} i
= \frac{ (n+1)^2 }{4}
\quad \text{and} \quad
\sum_{\substack{0 \leq i \leq n \\ i \text{ odd}}} i^2
= \frac{n(n+1)(n+2)}{6}
$.
\item
When $n$ is even, 
$\displaystyle
\sum_{\substack{0 \leq i \leq n \\ i \text{ even}}} i
= \frac{ n(n+2) }{4}
\quad \text{and} \quad
\sum_{\substack{0 \leq i \leq n \\ i \text{ even}}} i^2
= \frac{n(n+1)(n+2)}{6}
$.
\end{enumerate}
\end{lemma}

\section{Basic results on polygon diagrams}

\subsection{Base case pruned enumerations}
We start by working out $Q_{g,n}$ for some small values of $(g,n)$. 
\begin{proposition}\label{basecases}
\begin{align*}
Q_{0,1}(\mu_1) &= \delta_{\mu_1,0} \\
Q_{0,2}(\mu_1,\mu_2) &= \overline{\mu}_1 \delta_{\mu_1,\mu_2} \\
Q_{0,3}(\mu_1,\mu_2,\mu_3)& = 
\begin{cases}
2\mu_1\mu_2\mu_3, & \mu_1,\mu_2,\mu_3 > 0 \\
\mu_1\mu_2, &  \mu_1,\mu_2 > 0, \mu_3 = 0 \\
\overline{\mu}_1, & \mu_1 \text{ even}, \mu_2=\mu_3=0 \\
0, & \mu_1 \text{ odd}, \mu_2 = \mu_3=0 
\end{cases} \\
\end{align*}
\end{proposition}
Here $\delta$ is the Kronecker delta and $\overline{n} = n + \delta_{n,0}$ is as in \cite{DKM2017}: for a positive integer $\overline{n} = n$, and $\overline{0} = 1$.

\begin{proof}
On the disc, every edge is boundary parallel. Therefore $Q_{0,1}(\mu_1) = 0$ for all positive $\mu_1$.

For $(g,n)=(0,2)$, all non-boundary parallel edges must run between the two boundary components $B_1$ and $B_2$, and are all 
parallel to each other. A pruned polygon diagram must consist of a number of pairwise parallel bigons running between $F_1$ and $F_2$. Therefore $Q_{0,2}(\mu_1,\mu_2) = 0$ if $\mu_1\neq \mu_2$. If $\mu_1=\mu_2 > 0$, consider 
the bigon containing the decorated marked point on $F_1$. The location of its other vertex on $B_2$ uniquely determines the pruned polygon diagram. Therefore $Q_{0,2}(\mu_1,\mu_1) = \mu_1$, or 
$Q_{0,2}(\mu_1,\mu_1) = \overline{\mu}_1$ if we include the trivial case $Q_{0,2}(0,0)=1$.

For $(g,n)=(0,3)$, we can embed the pair of pants in the plane, with its usual orientation, and denote the three boundary components by $F_1 = F_{\text{outer}}$, $F_2 = F_{\text{left}}$ and $F_3 = B_{\text{right}}$, with $\mu_1$, $\mu_2$ and $\mu_3$ marked points respectively. Without loss of generality assume $\mu_1\geq \mu_2, \mu_3$.
A non-boundary parallel edge can be separating, with endpoints on the same boundary component and cutting the surface into two annuli, or non-separating, with endpoints on different boundary components. See figure \ref{P03-1}.

\begin{figure}
\begin{center}
\includegraphics[scale=0.7]{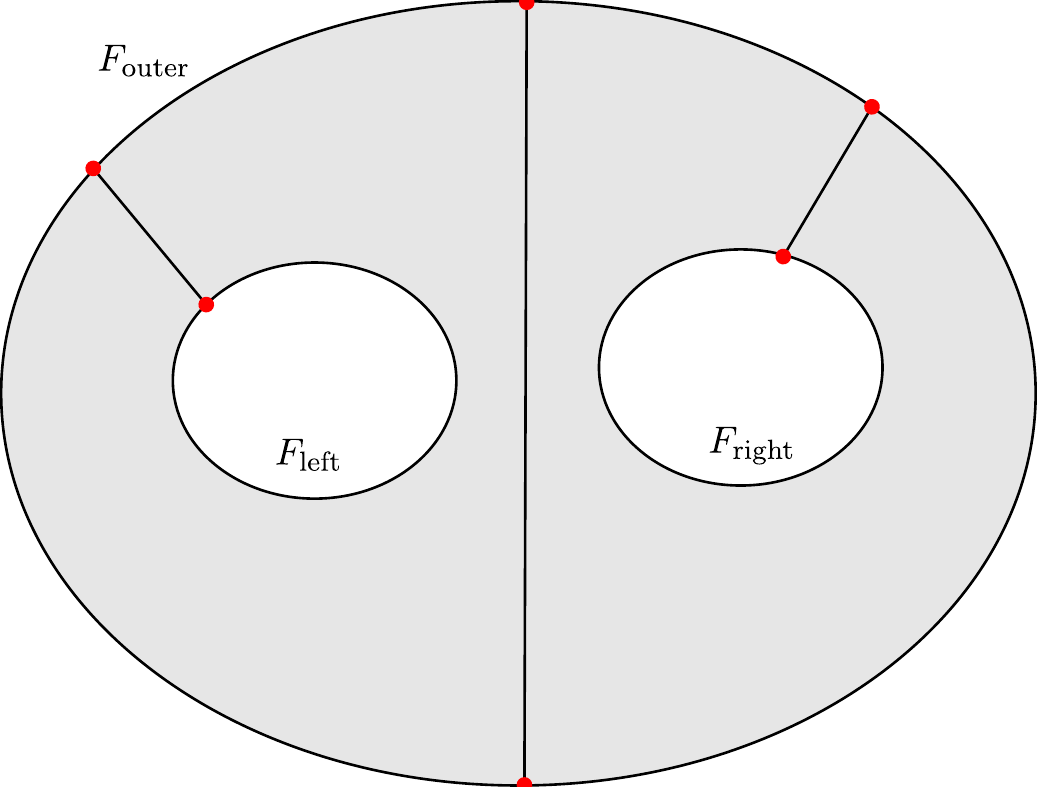}
\caption{Boundary labels and possible non-boundary parallel edges on a pair of pants.}
\label{P03-1}
\end{center}
\end{figure}

On a pair of pants there can be only one type of separating edge, and all separating edges must be parallel to each other. 
Consider a polygon $P$ in a pruned diagram. All its diagonals are also non-boundary parallel, for a boundary-parallel diagonal implies boundary-parallel edges. Further, $P$ cannot have more than one vertex on more than one boundary component; if there were two boundary components $F_i, F_j$ each with at least two vertices then there would be separating diagonals from each of $F_i, F_j$ to itself, impossible since there can be only one type of separating edge.
Moreover, $P$ cannot have three vertices on a single boundary component, since the three diagonals connecting them would have to be non-boundary parallel, hence separating, hence parallel to each other, hence forming a bigon at most.
Therefore a polygon in a pruned diagram on a pair of pants is of one of the following types:
\begin{itemize}
\item a non-separating bigon from one boundary component to another,
\item a separating bigon from one boundary component to itself,
\item a triangle with a vertex on each boundary component,
\item a triangle with two vertices on a single boundary component, and the third vertex on a different boundary component,
\item a quadrilateral with two opposite vertices on a single boundary component, and one vertex on each of the other two boundary components.
\end{itemize} 
See figure \ref{pants-polygons}.
It's easy to see that there can be at most one quadrilateral or two triangles in any pruned diagram. 

\begin{figure}
\begin{center}
\[
\begin{array}{cc}
\includegraphics[scale=0.65]{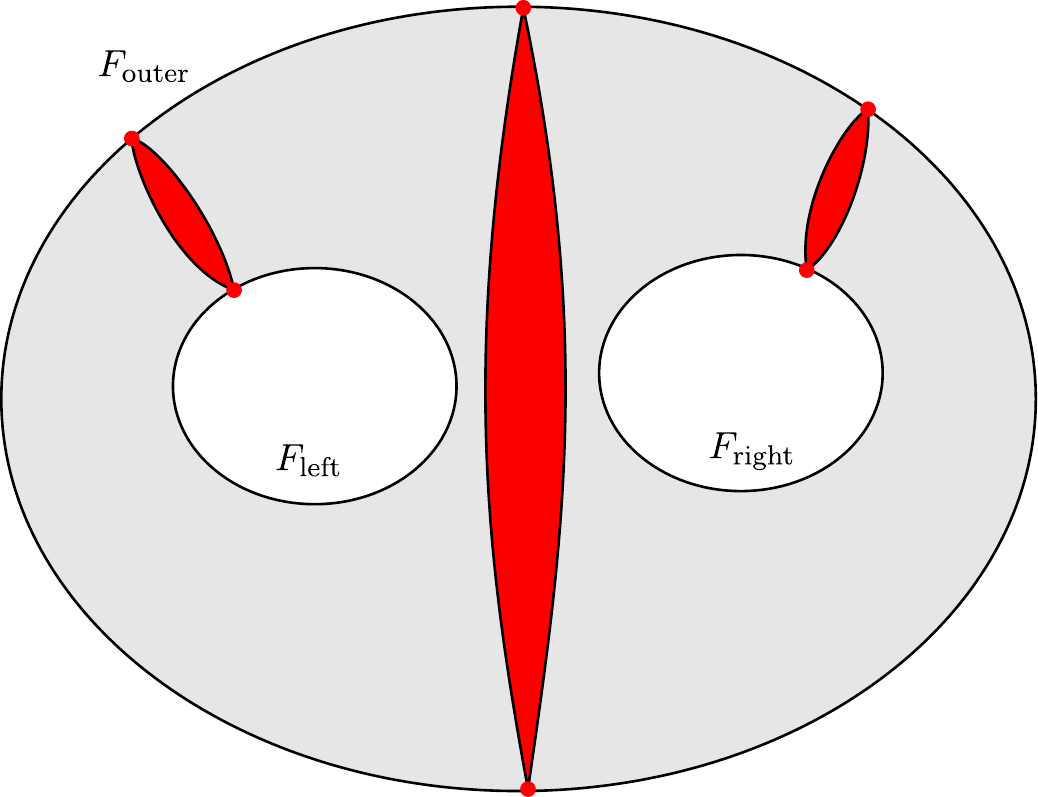} &
\includegraphics[scale=0.65]{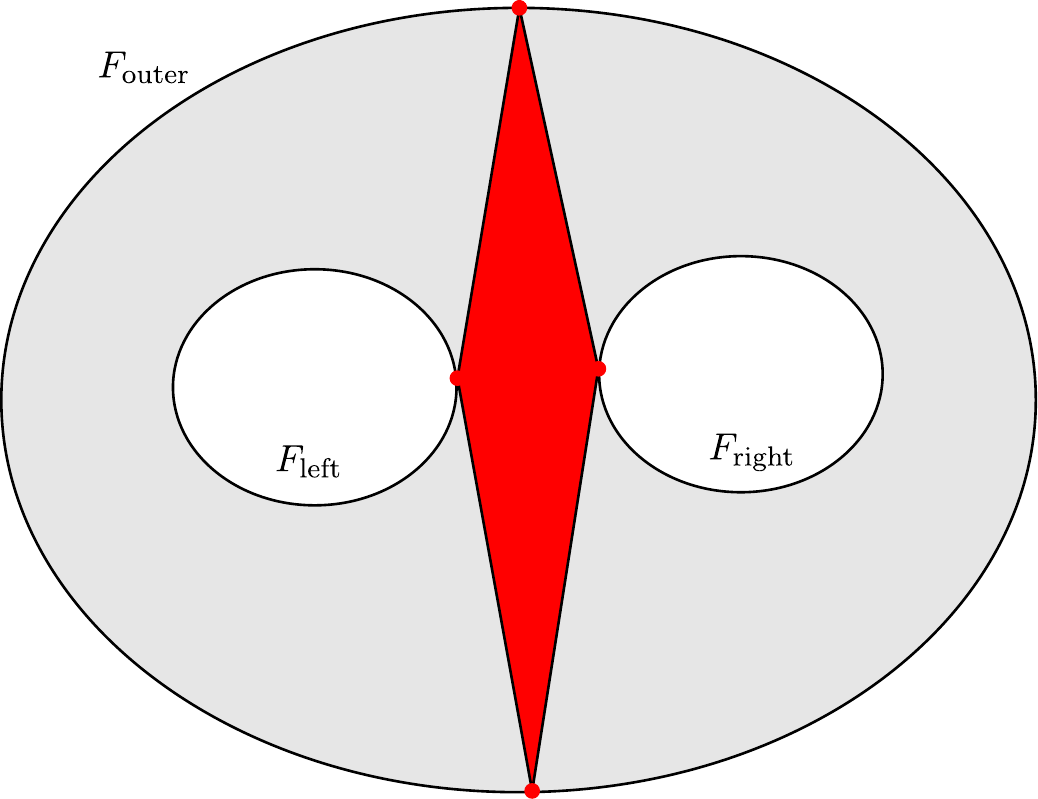} \\
\includegraphics[scale=0.65]{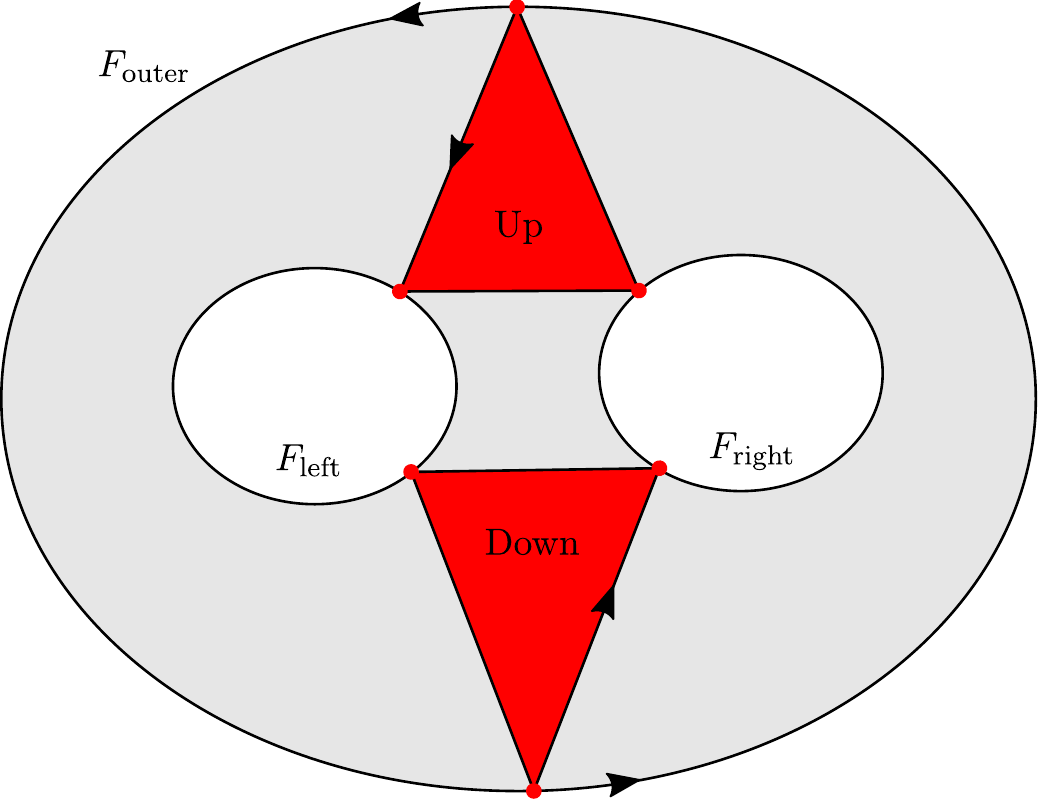} &
\includegraphics[scale=0.65]{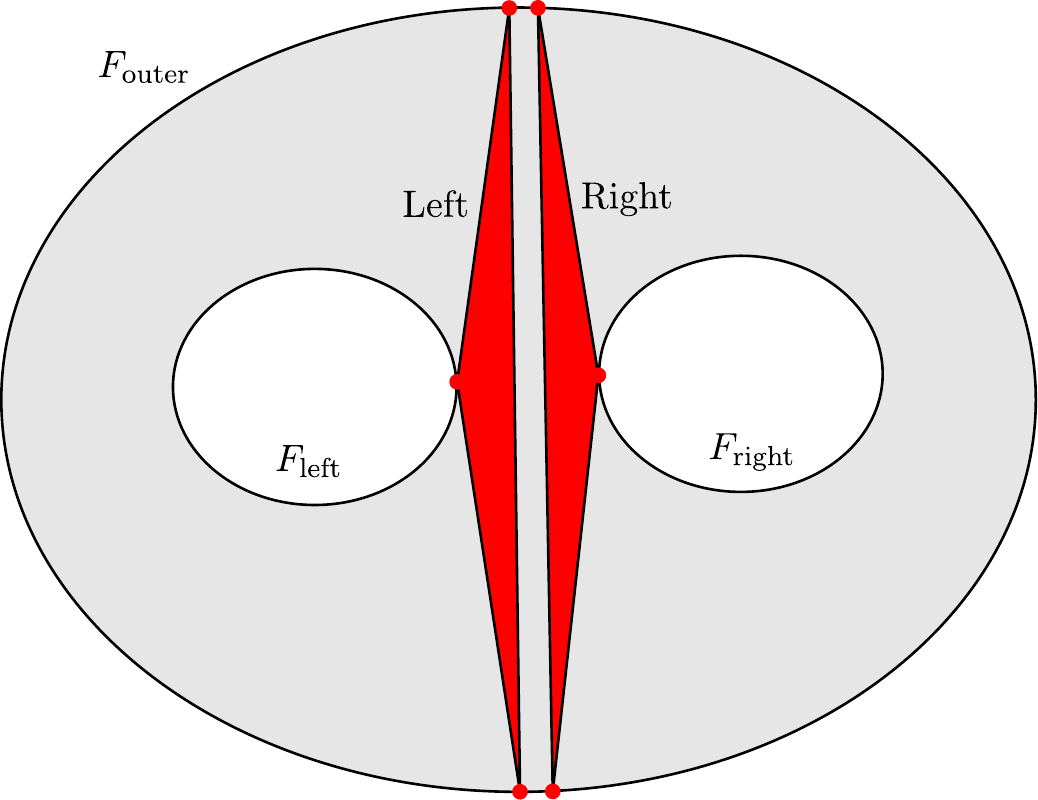}
\end{array}
\]
\caption{The decomposition of a polygon diagram.}
\label{pants-polygons}
\end{center}
\end{figure}

If $\mu_2=\mu_3=0$, then all edges must be between $B_{\text{outer}}$ and itself and separating. A pruned polygon diagram must consist of a number of pairwise
parallel bigons. Hence $Q_{0,3}(\mu_1,0,0)=0$ if $\mu_1$ is odd. If $\mu_1>0$ is even, then the configuration of $\frac{\mu_1}{2}$ separating bigons gives rise to
$\mu_1$ pruned polygon diagrams, as the decorated marked point can be located at any one of the $\mu_1$ positions. If $\mu_1 = 0$ then there is only the empty diagram, so in general there are $\overline{\mu}_1$ diagrams.

If $\mu_2 > 0$ and $\mu_3 = 0$, then since $\mu_1\geq \mu_2$, the possible polygons are 
\begin{itemize}
\item a non-separating bigon between $F_{\text{outer}}$ and itself,
\item a separating bigon between $F_{\text{outer}}$ and $F_{\text{left}}$,
\item a triangle with two vertices on $F_{\text{outer}}$ and a vertex on $F_{\text{left}}$.
\end{itemize}
Furthermore there can be at most one triangle. If $\mu_1-\mu_2$ is even, then a pruned polygon diagram must consist of $\mu_2$ bigons
from $F_{\text{outer}}$ to $F_{\text{left}}$ and $\frac{\mu_1-\mu_2}{2}$ bigons from $F_{\text{outer}}$ to itself. If $\mu_1-\mu_2$ is odd,
then a pruned polygon diagram must consist of a single triangle, $\mu_2-1$ bigons
from $F_{\text{outer}}$ to $F_{\text{left}}$ and $\frac{\mu_1-\mu_2-1}{2}$ bigon from $F_{\text{outer}}$ to itself. Again each such
configuration determines $\mu_1\mu_2$ pruned diagrams accounting for the locations of the two decorated marked points on $F_{\text{outer}}$ 
and $F_{\text{left}}$.

If $\mu_1, \mu_2,\mu_3 > 0$, then because $\mu_1$ is maximal, any separating edge or separating diagonal in a quadrilateral
must be from $F_{\text{outer}}$ to itself. Therefore the single quadrilateral (if it exists) must have a pair of opposite vertices on $F_{\text{outer}}$ and one vertex each on $F_{\text{left}}$ and $F_{\text{right}}$. There are two types of triangles with a separating edge from $F_{\text{outer}}$ to itself, depending
on whether the last vertex is on $F_{\text{left}}$ or $F_{\text{right}}$. Call these \emph{left} or \emph{right} triangles respectively. 
There are also two types of triangles with a vertex on each boundary component, depending on whether the triangle's boundary, inheriting an orientation from the surface, goes from $F_{\text{outer}}$ to $F_{\text{left}}$ or $F_{\text{right}}$. Call these \emph{up} or \emph{down} triangles respectively.  We then have the following cases.

\begin{enumerate}[label=(\roman*)]
\item
There is one quadrilateral. Then the pruned diagram must consist of this single quadrilateral,
$\mu_2-1$ bigons between $F_{\text{outer}}$ and $F_{\text{left}}$, and 
$\mu_3-1$ bigons between $F_{\text{outer}}$ and $F_{\text{right}}$. In this case we have $\mu_1 - \mu_2 - \mu_3 = 0$.
\item
There is a left and a right triangle. Then the pruned diagram must consist of these two triangles, $\mu_2-1$ bigons between $F_{\text{outer}}$ and $F_{\text{left}}$, $\mu_3-1$ bigons between $F_{\text{outer}}$ and $F_{\text{right}}$, and $\frac{\mu_1-\mu_2-\mu_3-2}{2}$ separating bigons between $F_{\text{outer}}$ and itself. In this case we have $\mu_1 - \mu_2 - \mu_3$ is positive and even.
\item 
There is an up and a down triangle. Then the pruned diagram must consist of these two triangles, $\frac{\mu_1+\mu_2-\mu_3-2}{2}$ 
bigons between $F_{\text{outer}}$ and $F_{\text{left}}$, $\frac{\mu_1+\mu_3-\mu_2-2}{2}$ bigons between $F_{\text{outer}}$ and $F_{\text{right}}$, and $\frac{\mu_2+\mu_3-\mu_1-2}{2}$ bigons between $F_{\text{left}}$ and $F_{\text{right}}$. In this case we have $\mu_1 - \mu_2 -\mu_3$ is negative and even. 
(Note that $\mu_1+\mu_2 - \mu_3$ and $\mu_1+\mu_3 - \mu_2$ are both positive and even in this case.)
\item 
There is a single left (resp. right) triangle. Then the pruned diagram must consist of this triangle, $\mu_2-1$ (resp. $\mu_3-1$) 
bigons between $F_{\text{outer}}$ and $F_{\text{left}}$ (resp. $F_{\text{right}}$), $\mu_3$ (resp. $\mu_2$) bigons between $F_{\text{outer}}$ and $F_{\text{right}}$ (resp. $F_{\text{left}}$), and $\frac{\mu_1-\mu_2-\mu_3-1}{2}$ separating bigons between $F_{\text{outer}}$ and itself. In this case $\mu_1 - \mu_2-\mu_3$ is positive and odd.
\item
There is a single up (resp. down) triangle. Then the pruned diagram must consist of this triangle, $\frac{\mu_1+\mu_2-\mu_3-1}{2}$ 
bigons between $F_{\text{outer}}$ and $F_{\text{left}}$, $\frac{\mu_1+\mu_3-\mu_2-1}{2}$ bigons between $F_{\text{outer}}$ and $F_{\text{right}}$, and $\frac{\mu_2+\mu_3-\mu_1-1}{2}$ bigons between $F_{\text{left}}$ and $F_{\text{right}}$. In this case $\mu_1 - \mu_2-\mu_3$ is negative and odd.
(Note that $\mu_1+\mu_2 - \mu_3$ and $\mu_1+\mu_3 - \mu_2$ are both positive and odd in this case.)
\item 
There are only non-separating bigons. Then the pruned diagram must consist of $\frac{\mu_1+\mu_2-\mu_3}{2}$ 
bigons between $F_{\text{outer}}$ and $F_{\text{left}}$, $\frac{\mu_1+\mu_3-\mu_2}{2}$ bigons between $F_{\text{outer}}$ and $F_{\text{right}}$, and $\frac{\mu_2+\mu_3-\mu_1}{2}$ bigons between $F_{\text{left}}$ and $F_{\text{right}}$. In this case $\mu_1 - \mu_2-\mu_3$ is negative or zero, and even. 
(Note that $\mu_1+\mu_2 - \mu_3$ and $\mu_1+\mu_3 - \mu_2$ are both positive and even in this case.)
\item
There are only bigons, some of which are separating. Then the pruned diagram must consist of $\mu_2$ bigons between $F_{\text{outer}}$ and $F_{\text{left}}$, $\mu_3$ bigons between $F_{\text{outer}}$ and $F_{\text{right}}$, and $\frac{\mu_1-\mu_2-\mu_3}{2}$ separating bigons between $F_{\text{outer}}$ and itself. In this case we have $\mu_1 - \mu_2 -\mu_3$ is positive and even.
\end{enumerate}
Observe that for each triple $(\mu_1,\mu_2,\mu_3)$, precisely two of these cases apply, depending on $\mu_1-\mu_2-\mu_3$. (Here we count the left and right versions of (iv) separately, and the up and down versions of (v) separately.) We thus have two possible configurations of polygons, and each configuration corresponds to $\mu_1\mu_2\mu_3$ pruned diagrams, accounting for the locations of the decorated marked points on the three boundary components. Thus $Q_{0,3}$ is as claimed.
\end{proof}

\subsection{Cuff diagrams}

Consider the annulus embedded in the plane with $F_1$ being the outer and $F_2$ the inner boundary. A {\em cuff diagram} is a polygon diagram on an annulus with no edges between vertices on 
the inner boundary $F_2$. (These correspond to the local arc diagrams of \cite{DKM2017}.) Let $L(b,a)$ be the number, up to equivalence, of cuff diagrams with $b$ vertices on the
outer boundary $F_1$ and $a$ vertices on the inner boundary $F_2$. 

\begin{proposition}\label{cuffcount}
\[
L(b,a) = 
\begin{cases}
a\binom{2b}{b-a}, &  a,b > 0\\
\frac{1}{2}\binom{2b}{b}, &a = 0, b >0\\
1, &  a=b=0 \\
\end{cases}
\]
\end{proposition}

\begin{proof}
This argument follows \cite{DKM2017}, using ideas of Przytycki \cite{Przytycki1999}. A {\em partial arrow diagram} on a circle is a labeling of a subset of 
vertices on the boundary of the circle with the label ``out". 

Assume $a > 0$. We claim there is a bijection between the set of equivalence classes
cuff diagrams counted by $L(b,a)$, on the one hand, and on the other, the set of partial arrow diagrams on a circle with $2b$ vertices and $b-a$ ``out" labels, together with a choice of
decorated marked point on the inner circle. Clearly the latter set has cardinality $a\binom{2b}{b-a}$.

This bijection is constructed as follows.
Starting from a cuff diagram $D$, observe that there are $b-a$ edges of $D$ with both endpoints on the outer boundary $F_1$. Orient these edges in an anticlockwise direction. (Note this orientation may disagree with the orientation induced from polygon boundaries.)  Label the $b$ vertices on $F_1$ from $1$ to $b$ starting from the decorated marked point. Taking a slightly smaller outer circle $F'_1$ close to $F_1$, the edges of $D$ intersect $F'_1$ in $2b$ vertices, 
say $1,1',2,2',\ldots, b,b'$. Label each of these $2b$ vertices ``out" if it is a starting point of one of the oriented edges. 
We then have $b-a$ ``out" labels, and hence a partial arrow diagram of the required type. 
The decorated marked point on the inner circle is given by the cuff diagram.

Conversely, starting from a partial arrow diagram, there is a unique way to reconstruct the edges of the cuff diagram $D$ so that 
they do not intersect. Regard the circle with $2b$ vertices of the partial arrow diagram as the outer boundary $F_1$, with the $2b$ vertices lying in pairs close to each marked point of the original annulus, and with the pair close to marked point $i$ labelled $i,i'$. 
Since there are both labelled and unlabelled vertices among the $2b$ vertices, there is an ``out" vertex on $F_1$ followed by an unlabelled vertex in a anticlockwise direction. The edge starting from this ``out" vertex must end at that neighbouring unlabelled vertex (otherwise
edges ending at those two vertices would intersect). Next we remove those two matched vertices and repeat the argument. Eventually 
all $b-a$ ``out" vertices are matched with unlabelled vertices by $b-a$ oriented edges. The remaining $2a$ unlabelled vertices 
are joined to $2a$ vertices on the inner circle $F_2$. These $2a$ edges divide the annulus into $2a$ sectors, which are further subdivided
into a number of disc regions by the oriented edges. Since $2a$ is even, the disc regions can be alternately coloured black and white. Each pair of vertices on $F_1$ is then pinched into the original marked point; the colouring can be chosen so that the pinched vertices are corners of black polygons near $F_1$. The vertices of $F_2$ can then be pinched in pairs in a unique way to produce a polygon diagram $D$, where the polygons are the black regions. This $D$ has $b$ vertices on $F_1$ and $a$ vertices on $F_2$. Finally, each vertex on $F_2$ belongs to a separate polygon with all other vertices on the outer circle. Placing the decorated marked point on $F_2$ at each vertex gives a distinct cuff diagram of the required type. See figure \ref{arrowdiagram2}.

\begin{figure}
\begin{center}
\includegraphics[scale=0.7]{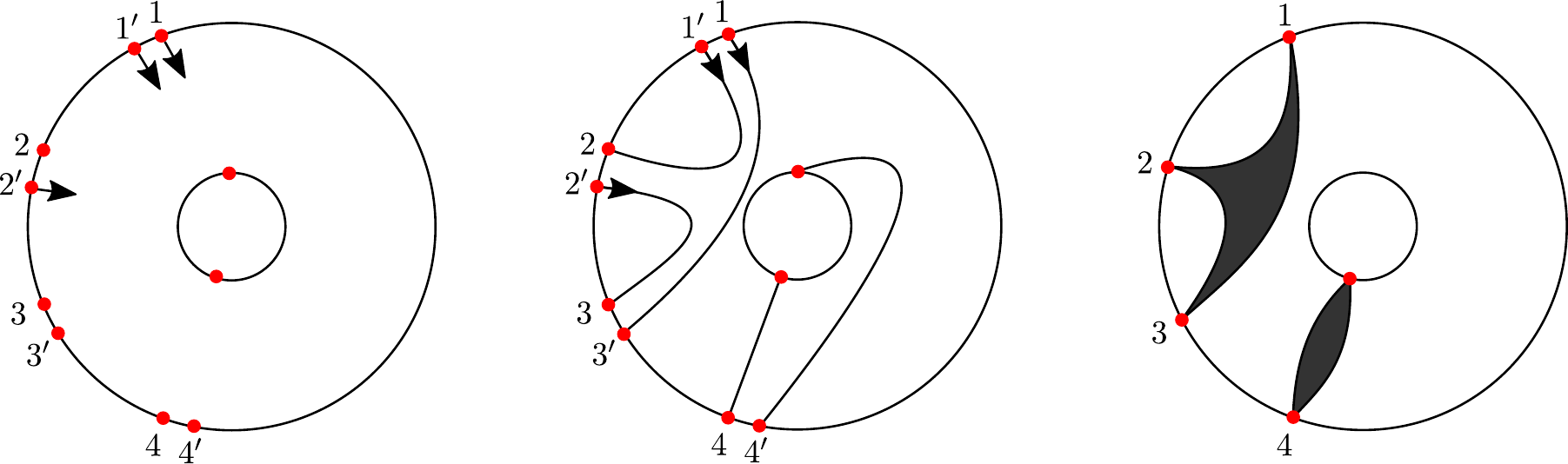}
\caption{Reconstructing a cuff diagram from a partial arrow diagram.}
\label{arrowdiagram2}
\end{center}
\end{figure}

If $a=0$ then the bijection fails. From the cuff diagram we can still construct a partial arrow diagram. But when the cuff diagram is being reconstructed from a partial arrow diagram, there is a single
non-disc region, so not every partial arrow diagram gives rise to a cuff diagram. 
Call a partial arrow diagram {\em compatible} if it yields a cuff diagram. Since each edge is now 
separating, the regions divided by the edges can still be alternately coloured black and white. 
All regions are discs except one which is an annulus. Again choose the colouring so that the pairs of vertices labelled $i,i'$ on $F_1$ are pinched into corners of 
black regions. The partial arrow diagram is then compatible if and only if the annulus region is white.
However, when the partial arrow diagram is not compatible, pinching instead the corners 
of white regions will then result in a cuff diagram. In other words, if we rotate all the ``out" labels by one spot 
counterclockwise, the new partial arrow diagram will be compatible. 
Conversely, if a partial arrow diagram is compatible, then rotating its labels one spot clockwise will result in an incompatible
partial arrow diagram. Hence there is a bijection between compatible and incompatible partial arrow diagrams, and the number of
cuff diagram is exactly half of the number of partial arrow diagrams, or $\frac{1}{2}\binom{2b}{b}$.

When $a=b=0$, there is the unique empty cuff diagram.
\end{proof}

\subsection{Annulus enumeration}

\begin{proposition}
\label{P02prop}
\begin{align*}
P_{0,2}(\mu_1,\mu_2)&=
\begin{cases} \binom{2\mu_1-1}{\mu_1}\binom{2\mu_2-1}{\mu_2}\left(\frac{2\mu_1\mu_2}{\mu_1+\mu_2}+1\right), &  \mu_1,\mu_2 > 0 \\ 
\binom{2\mu_1-1}{\mu_1}, &  \mu_2=0 
\end{cases} 
\end{align*}
\end{proposition}
\begin{proof}
If $\mu_2 = 0$ then a polygon diagram is just a cuff diagram, hence by proposition \ref{cuffcount}
\[
P_{0,2}(\mu_1,0)=L(\mu_1,0)=\frac{1}{2}\binom{2\mu_1}{\mu_1}=\binom{2\mu_1-1}{\mu_1}.
\]
Note that taking $\binom{-1}{0}=1$, this works even when $\mu_1=0$.

If $\mu_1, \mu_2 > 0$, then as we saw in the introduction, from \cite{MN2004} the number of connected polygon diagrams (i.e. with at least one edge from $F_1$ to $F_2$) is 
\[
\binom{2\mu_1-1}{\mu_1}\binom{2\mu_2-1}{\mu_2}\frac{2\mu_1\mu_2}{\mu_1+\mu_2}.
\]
If there are no edges between
the two boundaries, then the polygon diagram is a union of two cuff diagrams, hence
\begin{align*}
P_{0,2}(\mu_1,\mu_2) &=\binom{2\mu_1-1}{\mu_1}\binom{2\mu_2-1}{\mu_2}\frac{2\mu_1\mu_2}{\mu_1+\mu_2} + \frac{1}{2}\binom{2\mu_1}{\mu_1}\cdot\frac{1}{2}\binom{2\mu_2}{\mu_2} \\
&=\binom{2\mu_1-1}{\mu_1}\binom{2\mu_2-1}{\mu_2}\left(\frac{2\mu_1\mu_2}{\mu_1+\mu_2}+1\right)
\end{align*}
as required.
\end{proof}

\subsection{Decomposition of polygon diagrams}

Suppose $S$ is not a disc or an annulus. Then any polygon diagram on $S$ can be decomposed into a pruned polygon 
diagram on $S$ together with $n$ cuff diagrams, one for each boundary component of $S$. Take an annular collar 
of each boundary component of $S$, and isotope all boundary parallel edges to be inside the union of these annuli.
The inner circle of each annulus intersects the polygons in $\nu_i\geq 0$ arcs. 
Pinch each arc into a vertex, choose one vertex on each inner circle with $\nu_i > 0$ as a 
decorated marked point, and cut along each inner circle. This produces a cuff
diagram on each annular collar and a pruned polygon diagram on the shrunken surface. This decomposition
is essentially unique except for the choice of decorated marked points on the inner circles, i.e.,
a single polygon diagram will give rise to $\prod_{i=1}^n \overline{\nu}_i$ distinct decompositions. See figure \ref{pruned}.
Conversely, starting from such a decomposition, we can reconstruct the unique polygon diagram by
attaching the cuff diagrams to the pruned polygon diagram by identifying the corresponding decorated
marked points along the gluing circles, and unpinching all the vertices on the gluing circles into arcs. Therefore we have
the relationship between $P_{g,n}$ and $Q_{g,n}$, corresponding to the ``local decomposition" of arc diagrams in \cite{DKM2017}.

\begin{figure}
\begin{center}
\includegraphics[scale=0.65]{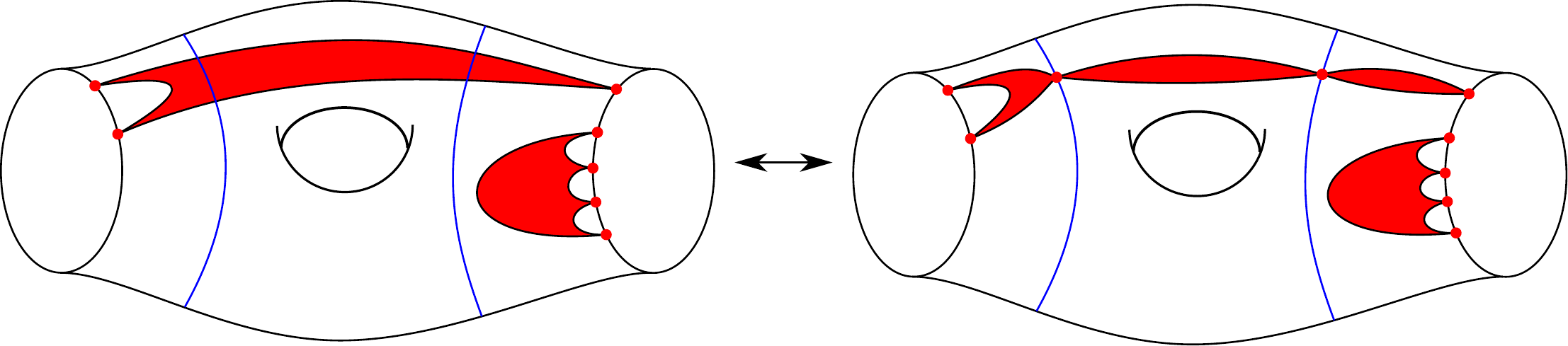}
\caption{The decomposition of a polygon diagram.}
\label{pruned}
\end{center}
\end{figure}

\begin{proposition}\label{PQ}
For $(g,n) \neq (0,1)$ or $(0,2)$,
\begin{equation}
\label{PQeq}
P_{g,n}(\mu_1, \ldots, \mu_n) = \sum_{0 \leq \nu_i \leq \mu_i} \left(Q_{g,n} (\nu_1, \ldots, \nu_n)\prod_{i=1}^n \frac{1}{\overline{\nu}_i}L(\mu_i, \nu_i)\right)
\end{equation}
\qed
\end{proposition}

It turns out that dividing by a power of 2 for each of the $\mu_i$ that is zero, we obtain a nicer form of this result, eliminating the piecewise nature of $L(\mu_i, \nu_i)$. The number of $\mu_i$ that are zero is given by $\sum_{i=1}^n \delta_{\mu_i,0}$. Defining
\[
{P}'_{g,n}(\mu_1, \ldots, \mu_n)=\frac{1}{2^{\sum^n_1 \delta_{\mu_i,0}}}P_{g,n}(\mu_1, \ldots, \mu_n) 
\quad \text{and} \quad
{Q}'_{g,n}(\nu_1, \ldots, \nu_n)=\frac{1}{2^{\sum^n_1 \delta_{\nu_i,0}}}Q_{g,n}(\nu_1, \ldots, \nu_n),
\]
and applying proposition \ref{cuffcount}, equation \eqref{PQeq} becomes
\begin{align}\label{PQ'}
P'_{g,n}(\mu_1, \ldots, \mu_n) = \sum_{0 \leq \nu_i \leq \mu_i} \left(Q_{g,n}'(\nu_1, \ldots, \nu_n)\prod_{i=1}^n \binom{2\mu_i}{\mu_i-\nu_i}\right).
\end{align}

\subsection{Pants enumeration}

\begin{proposition}
\label{P03count}
\[P_{0,3}(\mu_1,\mu_2,\mu_3)= \binom{2\mu_1-1}{\mu_1}\binom{2\mu_2-1}{\mu_2}\binom{2\mu_3-1}{\mu_3}\left(2\mu_1\mu_2\mu_3+\sum_{i\neq j}\mu_i\mu_j+\sum^3_{i=1} \frac{\mu_i^2-\mu_i}{2\mu_i-1}+1\right) \]
\end{proposition}
\begin{proof}
It is easier to work with $P'$ and $Q'$. We split the sum from \eqref{PQ'}
\[
P'_{0,3}(\mu_1, \mu_2, \mu_3) = \sum_{0\leq \nu_i \leq \mu_i} \left(Q_{0,3}'(\nu_1,\nu_2,\nu_3)\prod_{i=1}^3\binom{2\mu_i}{\mu_i-\nu_i}\right)
\]
into separate sums depending on how many of the $\nu_i$ are positive.
Using proposition \ref{basecases}, the sum over $\nu_i$ all being positive is given by
\begin{align*}
\sum_{\substack{0 \leq \nu_i \leq \mu_i \\ \text{all } \nu_i \text{ positive}}} Q_{0,3}'(\nu_1,\nu_2,\nu_3)\prod_{i=1}^3\binom{2\mu_i}{\mu_i-\nu_i}
=& \sum_{\substack{0 \leq \nu_i \leq \mu_i \\ \text{all } \mu_i \text{ positive}}}  2\nu_1\nu_2\nu_3\prod_{i=1}^3\binom{2\mu_i}{\mu_i-\nu_i}
= 2\prod_{i=1}^3 \sum_{1}^{\mu_i}\nu_i\binom{2\mu_i}{\mu_i-\nu_i}.
\end{align*}
Proposition \ref{almostpoly} then gives this expression as 
\[
2 \prod_{i=1}^3 \frac{\binom{2\mu_i}{\mu_i}}{2\mu_i-1} \left( P_0 (\mu_i) + Q_0 (\mu_i) \right)
= 2\prod_{i=1}^3\frac{\binom{2\mu_i}{\mu_i}}{2\mu_i-1}\frac{2\mu_i^2-\mu_i}{2} 
= \frac{\binom{2\mu_1}{\mu_1}}{2}\cdot\frac{\binom{2\mu_2}{\mu_2}}{2}\cdot\frac{\binom{2\mu_3}{\mu_3}}{2}\cdot(2\mu_1\mu_2\mu_3).
\]
Similarly, when $\nu_1 = 0$ and $\nu_2, \nu_3$ are positive we obtain
\begin{align*}
\sum_{\substack{0 \leq \nu_i \leq \mu_i \\ \nu_1 = 0, \nu_2,\nu_3>0 }} \left(Q'_{0,3} (\nu_1,\nu_2,\nu_3)\prod_{i=1}^3\binom{2\mu_i}{\mu_i-\nu_i}\right) =& 
\binom{2\mu_1}{\mu_1}\cdot \left(\sum_{\substack{0 \leq \nu_i \leq \mu_i \\ \nu_2,\nu_3 > 0 }} \left(\frac{1}{2}\nu_2\nu_3\prod_{i=2}^3\binom{2\mu_i}{\mu_i-\nu_i}\right)\right) \\
=&\frac{\binom{2\mu_1}{\mu_1}}{2}\cdot\frac{\binom{2\mu_2}{\mu_2}}{2}\cdot\frac{\binom{2\mu_3}{\mu_3}}{2}\cdot\left(\mu_2\mu_3\right).
\end{align*}
The sum over two $\nu_i$ being positive is given by repeating the above calculation with for $\nu_2 = 0$ and $\nu_3 = 0$. 
Continuing, when $\nu_1 = \nu_2 = 0$ and $\nu_3 > 0$ we obtain
\begin{align*}
\sum_{\substack{0 \leq \nu_i \leq \mu_i \\ \nu_1 = \nu_2= 0, \nu_3>0 }} \left(Q_{0,3}'(\nu_1,\nu_2,\nu_3)\prod_{i=1}^3\binom{2\mu_i}{\mu_i-\nu_i}\right) =& 
\binom{2\mu_1}{\mu_1}\cdot \binom{2\mu_2}{\mu_2}\cdot\left(\sum_{0<\nu_3\leq \mu_3, \nu_3 \text{ even} } \frac{1}{4}\nu_3\binom{2\mu_3}{\mu_3-\nu_3}\right) \\
=&\frac{\binom{2\mu_1}{\mu_1}}{2}\cdot\frac{\binom{2\mu_2}{\mu_2}}{2}\cdot\frac{\binom{2\mu_3}{\mu_3}}{2}\cdot\left(\frac{\mu_3^2-\mu_3}{2\mu_3-1}\right).
\end{align*}
The sum over one $\nu_i$ being positive is given by repeating the above calculation interchanging the roles of $\nu_1, \nu_2, \nu_3$. Finally when all $\nu_i$ are zero we have
\begin{align*}
\sum_{\substack{0 \leq \nu_i \leq \mu_i \\ \nu_1 = \nu_2 = \nu_3 = 0 }} 
\left(Q_{0,3}'(\nu_1,\nu_2,\nu_3)\prod_{i=1}^3\binom{2\mu_i}{\mu_i-\nu_i}\right) =& 
\frac{\binom{2\mu_1}{\mu_1}}{2}\cdot\frac{\binom{2\mu_2}{\mu_2}}{2}\cdot\frac{\binom{2\mu_3}{\mu_3}}{2}
\end{align*}
Note that with our convention of $\binom{-1}{0}=1$, $\binom{2\mu_i}{\mu_i} = 2^{\delta_{\mu_i,0}}\binom{2\mu_i-1}{\mu_i}$.
Summing the above terms, $P_{0,3} = 2^{\sum_{i=1}^n \delta_{\mu_i,0}} P'_{0,3}$ is given as claimed.
\end{proof}

\section{Recursions}
In this section we will prove recursion relations for both the polygon diagram counts $P_{g,n}$ and the
pruned polygon diagrams counts $Q_{g,n}$.
The recursion for $P_{g,n}$ is similar to that obeyed by the arc diagram counts $G_{g,n}$ in \cite{DKM2017}. The recursion for $Q_{g,n}$, appears messy at first sight, but if we only consider the dominant part, it 
actually differs very little from the recursion of non-boundary-parallel arc diagram count $N_{g,n}$ in \cite{DKM2017}. 
The top degree component of $N_{g,n}$ in turn agrees with the lattice count polynomials of Norbury, the volume polynomial of
Kontsevich, and the Weil-Petersson volume polynomials of Mirzakhani. 

We orient each boundary component $F_i$ as the boundary of $S$. This induces a cyclic order on the $\mu_i$ vertices on
$F_i$, and we denote by $\sigma(v)$ the next vertex to $v$ along $F_i$. If $\mu_i \geq 2$ then $\sigma(v)\neq v$. For any polygon diagram $D$, orient the edges of $D$ by choosing the orientation on each polygon to agree with the orientation on $S$.

\subsection{Polygon counts}

We now prove theorem \ref{thm:P_recursion}, the recursion on $P_{g,n}$, which states that for $g \geq 0$ and $\mu_1 > 0$, equation \eqref{precur-eq} holds:
\begin{align*}
P_{g,n} (\mu_1, \ldots, \mu_n)
&= P_{g,n} (\mu_1 - 1, \mmu_{X\setminus \{1\}}) + \sum_{k=2}^n \mu_k P_{g,n-1} ( \mu_1 + \mu_k - 1, \mmu_{X\setminus \{1,k\}} ) \\
& \quad + \mathop{\sum_{i+j=\mu_1 - 1}}_{j>0} \bigg[ P_{g-1,n+1} (i,j, \mmu_{X\setminus \{1\}}) + \mathop{\sum_{g_1 + g_2 = g}}_{I \sqcup J = X\setminus \{1\}} P_{g_1, |I|+1} (i, \mmu_I) \, P_{g_2, |J|+1} (j, \mmu_j) \bigg].
\end{align*}

\begin{proof}[Proof of theorem \ref{thm:P_recursion}]
Consider the decorated marked point $\mathbf{m}_1$ on the boundary component $F_1$. Suppose it is a vertex of the polygon $K$ of the diagram $D$. Let $\gamma$ be the outgoing edge from $\mathbf{m}_1$. If the other endpoint of $\gamma$ is also $\mathbf{m}_1$, then $K$ is a $1$-gon, and we obtain a new polygon diagram $D'$ by removing $K$ entirely (including $\mathbf{m}_1$), and then if $\mu_1\geq 2$, selecting the new decorated marked point on $F_1$ to be $\sigma(\mathbf{m}_1)$ (if $\mu_1=1$ then there will be no vertices on $F_1$ in $D'$, so we do not need a decorated marked point).
Conversely, starting with a polygon diagram $D'$ on $S_{g,n}$ with $(\mu_1-1,\mu_2,\ldots,\mu_n)$ boundary vertices, we can insert a $1$-gon on $F_1$ just before the decorated marked point $\mathbf{m}_1'$ (if there are no vertices on $F_1$, simply insert a $1$-gon), and then move the decorated marked point to the vertex of the new $1$-gon. These two operations are inverses of each other. This bijection gives the term $P_{g,n} (\mu_1 - 1, \mmu_{X\setminus \{1\}})$ in \eqref{precur-eq}.

If the other endpoint $v$ of $\gamma$ is different from $\mathbf{m}_1$, there are several cases.

\begin{description}
\item[(A) $\gamma$ has both endpoints on $F_1$ and is non-separating.]~

We cut $S = S_{g,n}$ along $\gamma$ into $S' = S'_{g-1,n+1}$, by removing a regular strip $\gamma \times (0,\epsilon)$ from $S$, where $\gamma = \gamma \times \{0\}$ and $\{\mathbf{m}_1\}\times [0,\epsilon]\subset F_1$ is a small sub-interval of $[\mathbf{m}_1,\sigma(\mathbf{m}_1))$. Then $F_1$ splits into two arcs, which together with $\gamma$ and a parallel copy $\gamma \times \{\epsilon\}$, form two boundary components $F'_0$ and
$F'_1$ on $S'$, with $\gamma$ part of $F'_1$. If $\sigma(\mathbf{m}_1) = v$ on $F_1$, then $F'_0$ contains no vertices. We obtain a polygon diagram $D'$ on $S'$ by collapsing $\gamma$ into a single vertex $\mathbf{m}'_1$ which is the decorated marked point on $F'_1$, and setting $\sigma(\mathbf{m}_1)$ as the decorated marked point 
on $F'_0$ (if there is at least one vertex on $F'_0$). The new diagram $D'$ has $i\geq 0$ vertices on $F'_0$ and $j\geq 1$ vertices on $F'_1$ with $i+j=\mu_1-1$.
Conversely starting with such a polygon diagram $D'$ on $S_{g-1,n+1}$ with $(i,j,\mu_2,\ldots,\mu_n)$ boundary vertices, we can reconstruct $D$. First expand the
decorated marked point $\mathbf{m}'_1$ on $F'_1$ into an interval. Then glue a strip joining this interval on $F'_1$ to an interval just before the decorated marked point on $F'_0$. (If $i=0$, we can glue to any interval on on $F'_0$.)  This bijection gives the term 
$\sum_{\substack{i+j=\mu_1 - 1, \  j>0}} P_{g-1,n+1} (i,j,\mmu_{X\setminus \{1\}})$ in \eqref{precur-eq}.

\item[(B) $\gamma$ has both endpoints on $F_1$ and is separating.]~

This is almost the same as the previous case. As before, we cut $S_{g,n}$ along $\gamma$ into two surfaces $S'_1$ and $S'_2$ 
with polygon diagrams $D'_1$ and $D'_2$, such that the
new vertex $\mathbf{m}'_1$ obtained from collapsing $\gamma$ is on $S'_2$. The polygon diagram $D$ can be uniquely reconstructed from such a pair
$(D'_1,D'_2)$. This bijection gives the term 
$\sum_{\substack{i+j=\mu_1 - 1, \  j>0}} \sum_{g_1 + g_2 = g, \ I \sqcup J = X\setminus \{1\}} P_{g_1, |I|+1} (i, \mmu_I) \, P_{g_2, |J|+1} (j, \mmu_j)$ 
of \eqref{precur-eq}.

\item[(C) $\gamma$ has endpoints $\mathbf{m}_1$ on $F_1$ and $v$ on $F_k$, $k > 1$.]~

In this case $\gamma$ is necessarily non-separating. Cutting $S_{g,n}$ along $\gamma$ and collapsing $\gamma$ following a similar procedure results in a polygon diagram
$D'$ on a surface $S'_{g,n-1}$ with $\mu_1+\mu_k-1$ vertices on its new boundary component $F'_1$, and the collapsed vertex $\mathbf{m}'_1$ as the 
decorated marked point on $F'_1$. However this is not a bijection since the information about original location of the decorated 
marked point on $F_k$ (relative to $v$) is forgotten in $D'$. In fact the map $D \to D'$ is $\mu_k$-to-$1$.
The decorated marked point $\mathbf{m}_k$ can be placed in any of the $\mu_k$ locations (relative to $v$). All $\mu_k$ such polygon diagrams will give rise to
the same $D'$ after cutting  along $\gamma$. Taking the multiplicity $\mu_k$ into account gives the term
$\sum_{k=2}^n \mu_k P_{g,n-1} ( \mu_1 + \mu_k - 1, \mmu_{X\setminus \{1,k\}} )$ of \eqref{precur-eq}.
\end{description}
\end{proof}

\subsection{Pruned polygon counts}

The recursion for pruned polygon diagrams follows from a similar analysis. It is more tedious due to the
fact that after cutting  along an edge $\gamma$, some other edges may become boundary parallel, so more care is required.

We previously referred to $\overline{n}$ as $\overline{n} = n$ if $n$ is a positive integer, and $\overline{0} = 1$, following \cite{DKM2017}. We now introduce another notation of a similar nature.

\begin{definition}
\label{tilde_notation}
For an integer $\mu$, let $\widetilde{\mu} = \mu$ if $\mu$ is a positive
even integer, and $0$ otherwise. 
\end{definition}

\begin{theorem}\label{qrecursion}
For $(g,n)\neq (0,1),(0,2),(0,3)$, the number of pruned polygon diagrams satisfies the following recursion:
\begin{align}
&Q_{g,n}(\mu_1, \ldots, \mu_n) \nonumber  = 
 \sum_{\substack{i+j+m = \mu_1 \\ i\geq 1, j,m\geq 0}}m Q_{g-1,n+1}(i,j,\mmu_{X\setminus \{1\}}) + \frac{\widetilde{\mu}_1}{2}Q_{g-1,n+1}(0,0,\mmu_{X\setminus \{1\}}) \nonumber\\  
+&\sum_{\substack{\mu_k>0 \\ 2\leq k \leq n}}\left( \sum_{\substack{i+m = \mu_1+\mu_k \\ i\geq 1, m\geq 0}}m \mu_kQ_{g,n-1}(i,\mmu_{X\setminus \{1,k\}}) +
\widetilde{\sum_{\substack{i+x = \mu_1-\mu_k \\ i\geq 1, x\geq 0}}}x \mu_kQ_{g,n-1}(i,\mmu_{X\setminus \{1,k\}})
+ \mu_1 \mu_kQ_{g,n-1}(0,\mmu_{X\setminus \{1,k\}}) \right) \nonumber\\
+& \sum_{\substack{\mu_k=0 \\ 2\leq k \leq n}}\left(\sum_{\substack{i+m = \mu_1 \nonumber \\ i\geq 1, m\geq 0}}m Q_{g,n-1}(i,\mmu_{X\setminus \{1,k\}})+ \widetilde{\mu}_1Q_{g,n-1}(0,\mmu_{X\setminus \{1,k\}})\right)\nonumber \\
+& \sum_{\substack{g_1+g_2=g \\ I\sqcup J = X\setminus \{1\} \\ \text{No discs or annuli}}}\left(\sum_{\substack{i+j+m=\mu_1 \\ i\geq 1, j,m \geq0}}mQ_{g_1,|I|+1}(i,\mmu_{I})Q_{g_2,|J|+1}(j,\mmu_{J}) + \frac{\widetilde{\mu}_1}{2}Q_{g_1,|I|+1}(0,\mmu_{I})Q_{g_2,|J|+1}(0,\mmu_{J})\right) \label{qrecursion-eq}
\end{align}
\end{theorem}
Here ``no discs or annuli" means $(g_1,|I|+1)$ and $(g_2,|J|+1)$ cannot be $(0,1)$ or $(0,2)$. The tilde summation 
$\widetilde{\sum}$ is defined to be
\begin{align*}
\widetilde{\sum_{\substack{i+x = \mu_1-\mu_k \\ i\geq 1, x\geq 0}}}x \mu_kQ_{g,n-1}(i,\mmu_{X\setminus \{1,k\}})= 
\sum_{\substack{i+x = \mu_1-\mu_k \\ i\geq 1, x\geq 0}}x \mu_kQ_{g,n-1}(i,\mmu_{X\setminus \{1,k\}})
- \sum_{\substack{i+x = \mu_k-\mu_1 \\ i\geq 1, x\geq 0}}x \mu_kQ_{g,n-1}(i,\mmu_{X\setminus \{1,k\}})
\end{align*}
Note that when $\mu_1\geq \mu_k$ the second sum vanishes, otherwise the first sum vanishes.
\begin{proof}
Suppose $D$ is a pruned polygon diagram on $S$. Let $\gamma$ be the outgoing edge at the decorated marked point $\mathbf{m}_1$ on $F_1$. Since there is
no $1$-gon in $D$ (they are boundary parallel), the other endpoint $v$ of $\gamma$ is distinct from $\mathbf{m}_1$.
As in \cite{DKM2017}, there are three cases for $\gamma$: (A) it has both endpoints on $F_1$ and is non-separating; (B) it has endpoints on $F_1$ and some other $F_k$, or has both endpoints on $F_1$ and cuts off an annulus parallel to $F_k$; or (C) it has both ends on $F_1$, is separating, and does not cut off an annulus. Each of these cases, especially case (B), has numerous sub-cases, which we now consider in detail.
\begin{description}
\item[(A) $\gamma$ has both endpoints on $F_1$ and is non-separating.]~

If an edge becomes boundary parallel after cutting $S$ along $\gamma$, then it must be parallel to $\gamma$ on $S$ (relative to endpoints) to begin with. 
Given two edges $\beta_1$ and $\beta_2$, both parallel to $\gamma$, let $I$ be a strip bounded by $\beta_1$, $\beta_2$ and portions of $F_1$. 
This strip $I$ is unique, because after we cut open along $I$,
$\beta$ and $\beta'$ belong to different boundary components, so they cannot bound any other strips.
There is a unique minimal strip $A:[0,1]^2\to S$ containing all edges 
parallel to $\gamma$, given by the union of connecting strips between all pairs of edges parallel to $\gamma$.
The left (resp. right) boundary of $A$ is an edge 
$\gamma_L$ (resp. $\gamma_R$) joining two vertices $p_L$ and $q_L$ (resp. $p_R$ and $q_R$), and the bottom (resp. top) boundary of $A$ is an interval
on $F_1$ from $p_L$ to $p_R$ (resp. $q_R$ to $q_L$). 
Note that $A$ may be degenerate, i.e. $\gamma_L$ and $\gamma_R$ may have one or both of their endpoints in common, or 
they are the same edge $\gamma$.

\begin{figure}
\begin{center}
\includegraphics[scale=0.75]{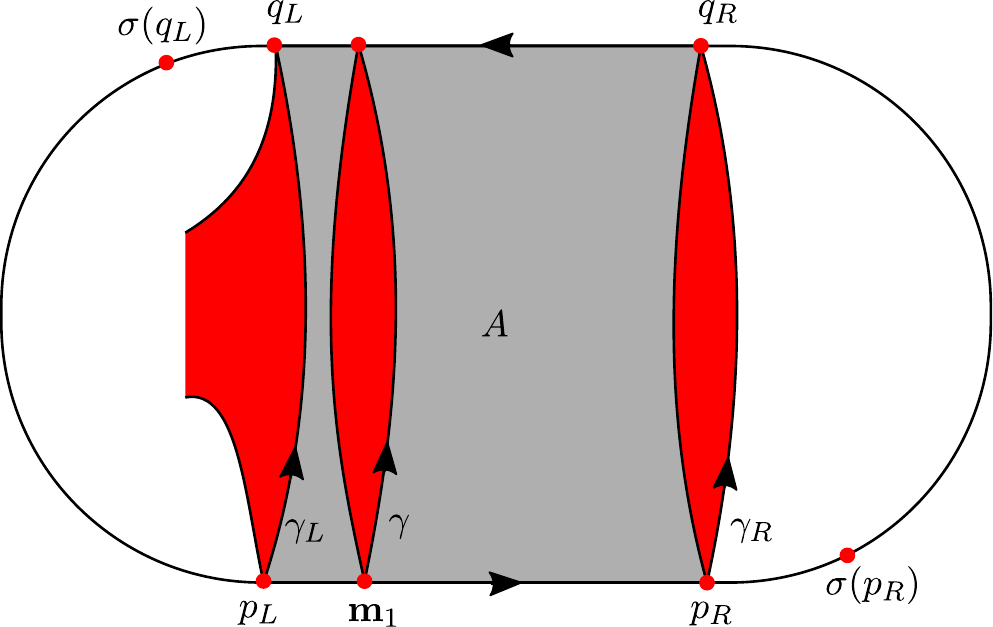}
\caption{Possible configurations of polygons in case (A).}
\label{caseA}
\end{center}
\end{figure}
 
Observe that all the edges in $A$, with the possible exception of $\gamma_L$ and $\gamma_R$, form a block of consecutive 
parallel bigons inside $A$.
Let there be $m\geq 1$ polygons with at least one edge parallel to $\gamma$. 
See figure \ref{caseA}.
There are four cases.
\begin{enumerate}[label=(\arabic*)]
\item All $m$ such polygons are bigons. In this case the $\mu_1$ vertices along $F_1$ are divided into
$4$ cyclic blocks of consecutive vertices: there is a block of $m$ vertices $(p_1,\ldots, p_m)$
followed by $j\geq 0$ vertices, followed by another block of $m$ consecutive vertices $(q_m,\ldots, q_1)$, followed
by $i\geq 0$ vertices, such that there is a bigon between each pair of vertices $\{p_i,q_i\}$, and $\mathbf{m}_1 \in \{p_1,\ldots, p_m\}$.
Remove all $m$ bigons from the pruned polygon diagram $D$ and cut $S$ along $\gamma$. If $j>0$ then let $\sigma(p_m)$ be
the decorated marked point on the new boundary component $F'_1$. If $i>0$ then let $\sigma(q_1)$ be
the decorated marked point on that new boundary component $F'_0$. This produces a pruned polygon diagram $D'$ on
$S'_{g-1,n+1}$ with $(i,j,\mu_2,\ldots,\mu_n)$ boundary vertices. The map $D\to D'$ is $m$-to-$1$, since $\mathbf{m}_1$ can
be any one of $\{p_1,\ldots, p_m\}$ and still produce the same pruned polygon diagram $D'$. Conversely $D$ can be
reconstructed for $D'$ up to the possible location of $\mathbf{m}_1$ as one of $\{p_1,\ldots, p_m\}$. 
Therefore we have the following contribution to \eqref{qrecursion-eq}:
\begin{equation}\label{eq1}
\sum_{\substack{i+j+2m = \mu_1 \\ m\geq 1, i,j\geq 0}}m Q_{g-1,n+1}(i,j,\mmu_{X\setminus \{1\}}).
\end{equation}
\item $\gamma_L$ is part of a polygon $K$ which is not a bigon, all other polygons are bigons. If $\gamma_L \neq \gamma_R$ then
$K$ and $A$ lie on the opposite sides of $\gamma_L$ (otherwise $K\subseteq A$, so must be a bigon), and there are $m-1$ bigons in $A$. 
Remove all bigons, cut $S$ along $\gamma_L$,
collapse $\gamma_L$ to a single vertex $\mathbf{m}'_0$ which we take to be the decorated marked point on the new boundary component $F'_0$, 
and let $\sigma(p_R)$ be the decorated marked point on $F'_1$. This produces a pruned polygon diagram $D'$. Similar to the previous case,
the map $D\to D'$ is $m$-to-$1$, as $\mathbf{m}_1$ can any one of the $m$ vertices between $p_L$ and $p_R$. Therefore we have the following 
contribution to \eqref{qrecursion-eq}:
\begin{equation}\label{eq2}
\sum_{\substack{i+j+2m = \mu_1 \\ m \geq 1, i,j\geq 0}}m Q_{g-1,n+1}(i+1,j,\mmu_{X\setminus \{1\}}) = 
\sum_{\substack{i+j+2m-1 = \mu_1 \\ i, m\geq 1,j\geq 0}}m Q_{g-1,n+1}(i,j,\mmu_{X\setminus \{1\}}).
\end{equation}
Note that this formula includes the contribution from the special case $\gamma_L = \gamma_R = \gamma$, where $m=1$.
\item $\gamma_R$ is part of a polygon $K$ which is not a bigon, and all other polygons are bigons. This is almost identical to the previous case,
except now $\gamma$ cannot be the edge $\gamma_R$. (If we had $\gamma = \gamma_R$ then, since $\gamma$ is the outgoing edge from $\mathbf{m}_1$, the polygon containing $\gamma$
would have to be on the same side of $\gamma$ as $A$.) The map $D \mapsto D'$ is now $(m-1)$-to-$1$, as $\mathbf{m}_1$ cannot be $p_R$.
Therefore we have the following contribution to \eqref{qrecursion-eq}:
\begin{equation}\label{eq3}
\sum_{\substack{i+j+2m = \mu_1 \\ m \geq 1, i,j\geq 0}}(m-1) Q_{g-1,n+1}(i,j+1,\mmu_{X\setminus \{1\}}) = 
\sum_{\substack{i+j+2m-1 = \mu_1 \\ j, m\geq 1,i\geq 0}}(m-1) Q_{g-1,n+1}(i,j,\mmu_{X\setminus \{1\}}).
\end{equation}
Note that this formula correctly excludes the special case $\gamma_L = \gamma_R = \gamma$, where $(m-1)=0$ and the formula vanishes.
\item $\gamma_L$ and $\gamma_R$ are each part of some polygon which is not a bigon, all other polygons are bigons. 
We allow $\gamma_L$ and $\gamma_R$ to be different edges of the same polygon. We obtain a pruned polygon diagram $D'$
by removing the $(m-2)$ bigons and collapsing $\gamma_L$ and $\gamma_R$ to 
decorated marked points $\mathbf{m}'_0$ and $\mathbf{m}'_1$. For the same reason as the 
previous case, $\gamma$ cannot be the edge $\gamma_R$, so the map $D\to D'$ is only 
$(m-1)$-to-$1$. Therefore the contribution to \eqref{qrecursion-eq} is
\begin{equation}\label{eq4}
\sum_{\substack{i+j+2m = \mu_1 \\ m\geq 1, i,j\geq 0}}(m-1) Q_{g-1,n+1}(i+1,j+1,\mmu_{X\setminus \{1\}}) = 
\sum_{\substack{i+j+2m = \mu_1 \\ i,j\geq 1, m\geq 0}}m Q_{g-1,n+1}(i,j,\mmu_{X\setminus \{1\}}).
\end{equation}
\end{enumerate}
Now we compute the total contribution from cases (A)(1)--(4). We drop the subscripts $g-1,n+1$ from $Q_{g-1,n+1}$ and $X \backslash \{1\}$ from $\mmu_{X \setminus \{1\}}$ for convenience. Summing expressions \eqref{eq1} and \eqref{eq4} and separating the terms according to where $i,j$ are zero or nonzero, we obtain
\begin{align}\label{eq100}
&\left( \sum_{\substack{i+j+2m = \mu_1 \\ m\geq 1, i,j\geq 0}} + \sum_{\substack{i+j+2m = \mu_1 \\ i,j\geq 1, m\geq 0}} \right) m Q (i,j,\mmu) \nonumber \\
&\quad =\sum_{\substack{i+j+2m = \mu_1 \\ i,j,m\geq 1}}2m Q(i,j,\mmu) + \sum_{\substack{j+2m = \mu_1 \\ j,m\geq 1}}m Q (0,j,\mmu) \nonumber 
+\sum_{\substack{i+2m = \mu_1 \\ i,m\geq 1}}m Q(i,0,\mmu) + \frac{\widetilde{\mu}_1}{2}Q (0,0,\mmu)\nonumber \\
&\quad = \sum_{\substack{i+j+2m = \mu_1 \\ i,m\geq 1, j\geq 0}}2m Q (i,j,\mmu) + \frac{\widetilde{\mu}_1}{2}Q (0,0,\mmu)
\end{align}
Similarly for expressions \eqref{eq2} and \eqref{eq3},
\begin{align}\label{eq101}
\sum_{\substack{i+j+2m-1 = \mu_1 \\ i, m\geq 1,j\geq 0}} & m Q(i,j,\mmu)+\sum_{\substack{i+j+2m-1 = \mu_1 \\ j, m\geq 1,i\geq 0}}(m-1) Q(i,j,\mmu)\nonumber \\
&=\sum_{\substack{i+j+2m-1 = \mu_1 \\ i,j,m\geq 1}}(2m-1) Q(i,j,\mmu)+
\sum_{\substack{i+2m-1 = \mu_1 \\ i,m\geq 1}}m Q(i,0,\mmu) \nonumber 
+\sum_{\substack{j+2m-1 = \mu_1 \\ j,m\geq 1}}(m-1) Q(0,j,\mmu)\nonumber \\
&=\sum_{\substack{i+j+2m-1 = \mu_1 \\ i,m\geq 1, j\geq 0}}(2m-1) Q(i,j,\mmu)
\end{align}
Adding \eqref{eq100} and \eqref{eq101} we have the first line of \eqref{qrecursion-eq}.

\item[(B) $\gamma$ has endpoints on $F_1$ and $F_k$, or has both endpoints on $F_1$ and cuts off an annulus parallel to $F_k$.] ~

Here $k \neq 1$. Note that since $(g,n)\neq (0,3)$, if $\gamma$ cuts off an annulus parallel to $F_k$, the remaining surface is not an annulus. Hence
different values of $k$ give different pruned polygon diagrams. There is no double counting when we sum over $k$.

To standardise the possibilities for $\gamma$, we define a path $\alpha$ from $F_1$ to $F_k$ as follows; $\bar{\alpha}$ denotes $\alpha$ with reversed orientation. If $\gamma$ has endpoints on $F_1$ and $F_k$, then let $\alpha = \gamma$. 
In this case, the edges that become parallel after $S$ is cut along $\gamma$ are precisely three types of curves: those parallel to the concatenated paths $\alpha$, $\alpha F_k \bar{\alpha}$, and $\bar{\alpha} F_1 \alpha$. On the other hand, if $\gamma$ has both endpoints on $F_1$ and cuts off an annulus parallel to $F_k$, then let $\alpha$ be a curve inside that annulus, connecting $F_1$ to $F_k$. In this case, the curves that become boundary parallel after $S$ is cut along $\gamma$ must be parallel to $\gamma$. See figure \ref{paths}.

\begin{figure}
\begin{center}
\includegraphics[scale=0.7]{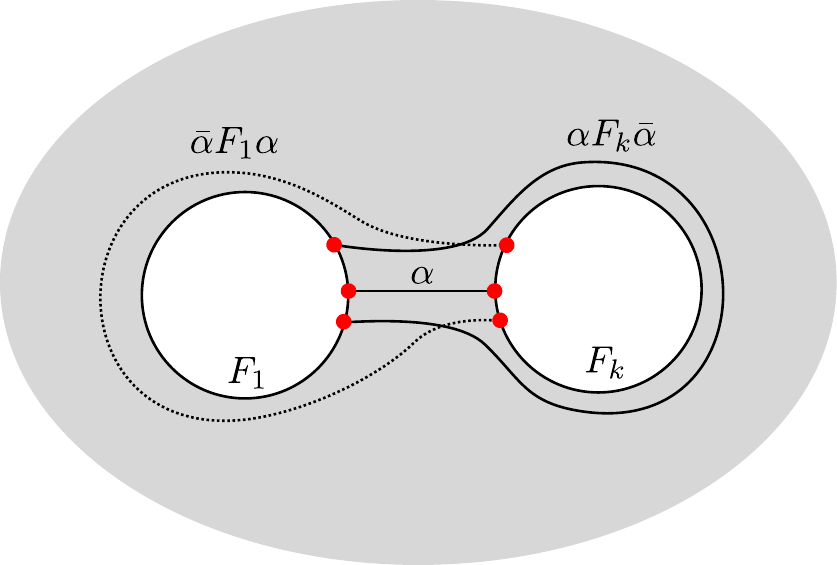}
\caption{The paths $\alpha$ and related paths in case (B).}
\label{paths}
\end{center}
\end{figure}

Since $S$ is not an annulus, there is a unique minimal strip $A^1$ containing all edges parallel to $\alpha$, 
bounded by edges $\gamma^1_L$ (resp. $\gamma^1_R$) joining two vertices $p^1_L\in F_1$ and $q^1_L\in F_k$ 
(resp. $p^1_R$ and $q^1_R$). The top (resp. bottom) boundary of $A^1$ is an interval
on $F_1$ (resp. $F_k$) from $p^1_L$ to $p^1_R$ (resp. $q^1_R$ to $q^1_L$). Similarly there are unique minimal strips $A^2$ 
and $A^3$ containing all edges of the second and third type respectively, with analogous notations.
Note that edges of the second and third types cannot appear simultaneously, so $A^2$ and $A^3$ cannot both be non-empty.
All three strips $A^i$ may be degenerate. See figure \ref{A_i_def}.

\begin{figure}
\begin{center}
\includegraphics[scale=0.65]{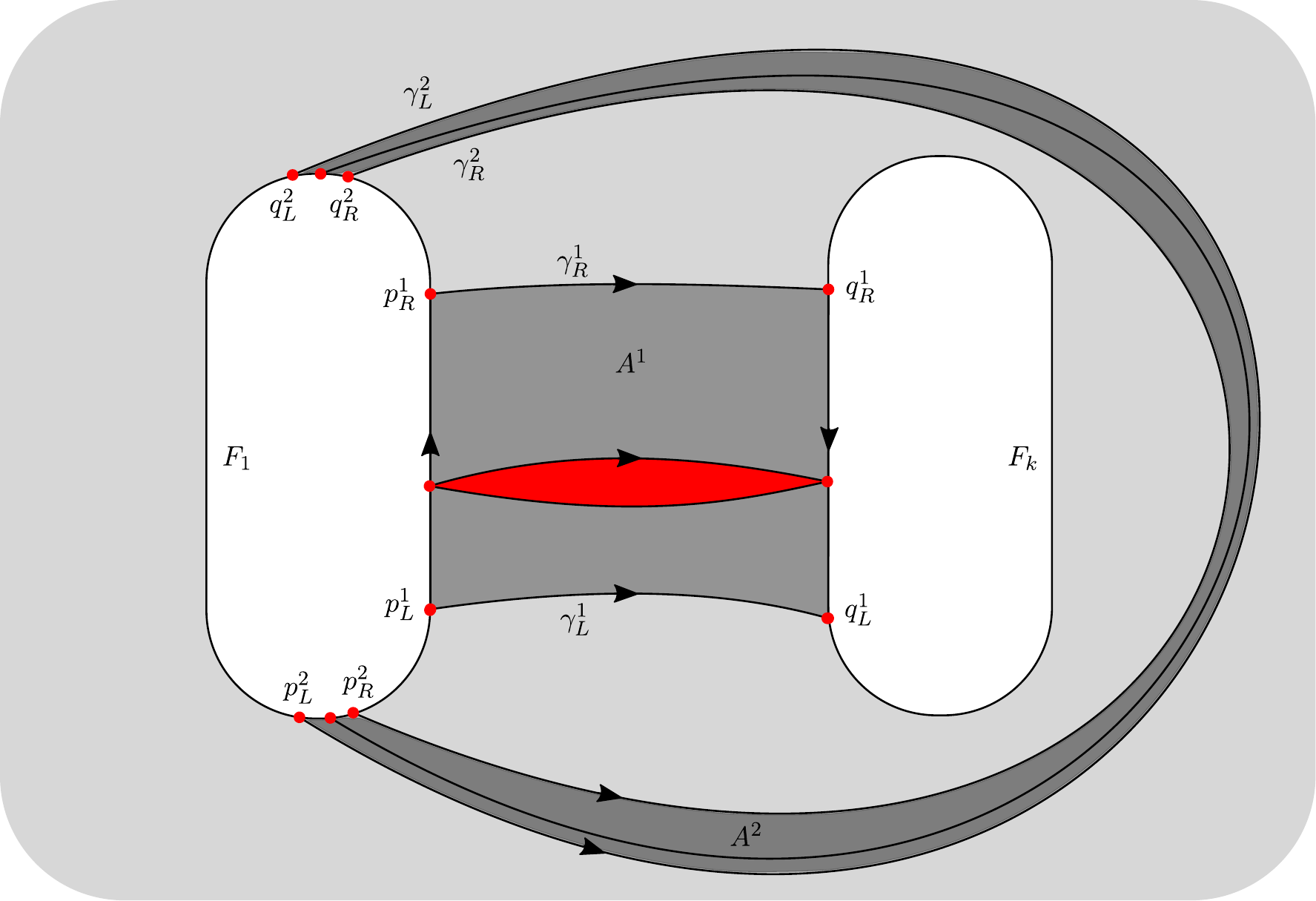}
\caption{The configurations of the strips $A^i$. In this figure $A^1, A^2$ are nonempty.}
\label{A_i_def}
\end{center}
\end{figure}

Call a polygon \emph{partially boundary parallel} if at least one of its edges is of the three types $\alpha, \alpha F_k \bar{\alpha}, \bar{\alpha} F_1 \alpha$. Call a polygon \emph{totally boundary parallel} if all of its edges are of these three types, and \emph{mixed} if
it is partially boundary parallel but not totally boundary parallel.
A totally boundary parallel polygon is either a bigon, or a triangle with 
two edges parallel to $\alpha$ and the third edge parallel to $\alpha F_k \bar{\alpha}$ or $\bar{\alpha} F_1 \alpha$.
Furthermore there can be at most one totally boundary parallel triangle.
Let there be $m$ partially boundary parallel polygons. Note $m \geq 1$, since $\gamma$ lies in a partially boundary parallel polygon.

Assume $\mu_k>0$. We split into the following sub-cases: all $m$ partially boundary parallel polygons are bigons; $m-1$ bigons and one totally boundary parallel triangle; there is a total boundary parallel triangle and a mixed polygon; there is a mixed polygon but no totally boundary parallel triangle.
\begin{enumerate}[label=(\arabic*)]
\item All $m$ partially boundary parallel polygons are bigons. We then split further into sub-cases accordingly as there are bigons parallel to $\alpha F_k \bar{\alpha}$ or $\bar{\alpha} F_1 \alpha$, or not.
\begin{enumerate}
\item There are no bigons parallel to $\alpha F_k \bar{\alpha}$ or $\bar{\alpha} F_1 \alpha$. Then there are $m$ consecutive 
bigons between $F_1$ and $F_k$. Removing all $m$ bigons and cutting $S$ along $\gamma$ gives a pruned polygon diagram $D'$ with 
$i=\mu_1+\mu_k - 2m$ vertices on the new boundary component $F'_1$. When $i>0$, the decorated marked point on $F'_1$ is set
to be $\sigma(p^1_R)$ if $\mu_1 > m$, and $\sigma(q^1_L)$ if $\mu_1 = m$. The map $D \mapsto D'$ is $m\mu_k$-to-$1$, since
$\mathbf{m}_1$ can be any of $m$ vertices of the bigons on $F_1$, and $\mathbf{m}_k$ can be any of the $\mu_k$ vertices on $F_k$.
Therefore we have the contribution
\begin{align}\label{eq111}
\sum_{\substack{i+2m = \mu_1+\mu_k \\ 1\leq m \leq \min(\mu_1,\mu_k), i\geq 0}}m {\mu}_kQ_{g,n-1}(i,\mmu_{X\setminus \{1,k\}}).
\end{align}
\item \label{case1b} There are $x\geq 1$ bigons parallel to $\alpha F_k \bar{\alpha}$. See figure \ref{fig:case1b}. Since $\alpha F_k \bar{\alpha}$ cuts off
an annulus parallel to $F_k$, the $\mu_k$ vertices on $F_k$ belong to $\mu_k$ bigons between $F_1$ and $F_k$.
Removing all $m = x + \mu_k$ bigons and cutting
along $\gamma$ gives a pruned polygon diagram $D'$ with $i=\mu_1 - m - x$ vertices on the new boundary component $F'_1$.
The decorated marked point on $F'_1$ is set to be $\sigma(q^1_L)$ if $i>0$. The map $D \mapsto D'$ is $(2x+\mu_k)\mu_k$-to-$1$, since
$\mathbf{m}_1$ can be any of the $(2x+\mu_k)$ vertices of the bigons on $F_1$. 
Therefore we have the contribution
\[
\sum_{\substack{i+2x = \mu_1-\mu_k \\ x\geq 1, i\geq 0}}(2x+\mu_k) {\mu}_kQ_{g,n-1}(i,\mmu_{X\setminus \{1,k\}}).
\]

\begin{figure}
\begin{center}
\includegraphics[scale=0.65]{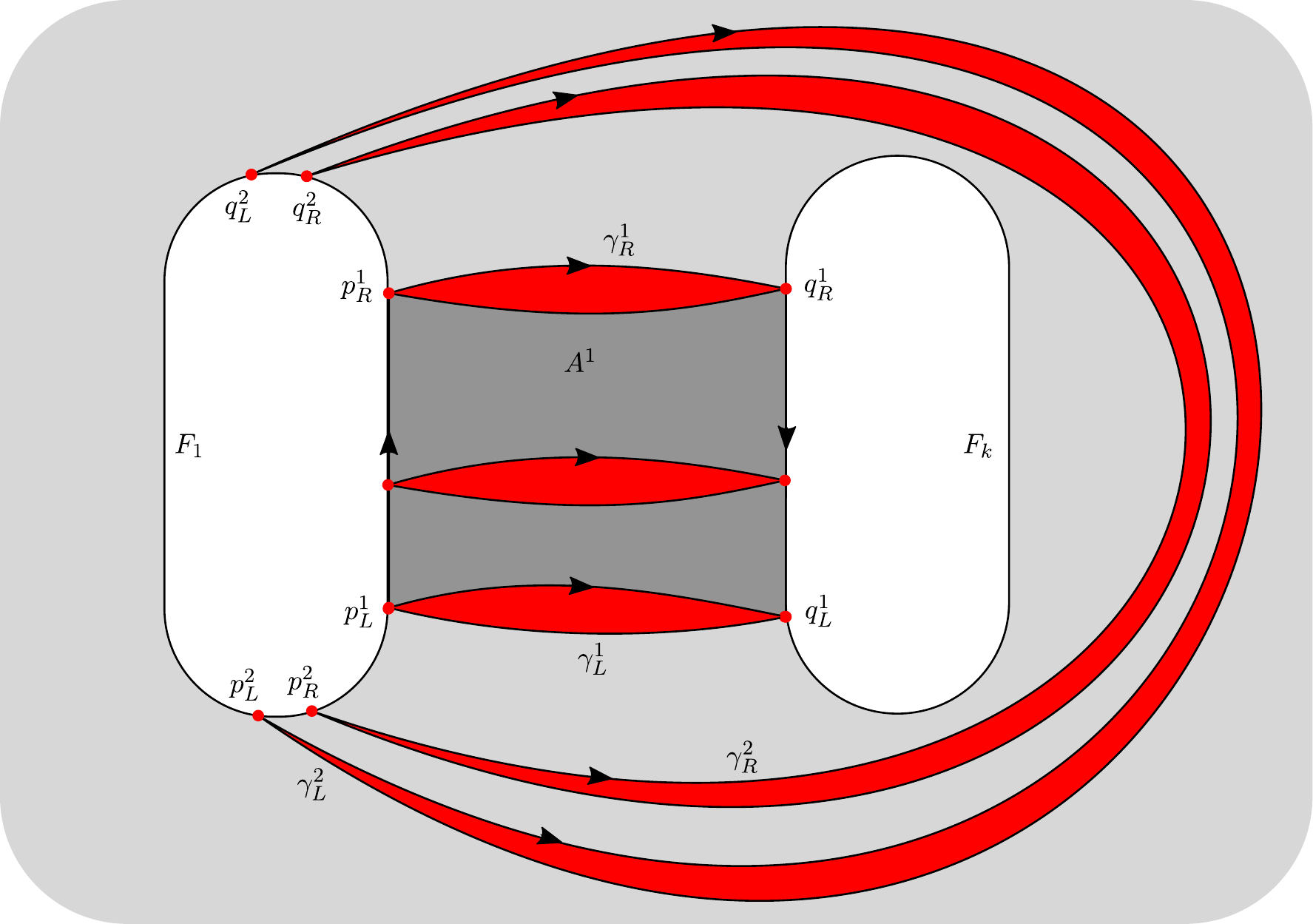}
\caption{Configuration of polygons in case (B)(1)(b).}
\label{fig:case1b}
\end{center}
\end{figure}

Splitting the sum in by writing $2x+\mu_k$ as $(x+\mu_k) + x$ and setting $m = x+\mu_k$, we note that $i+2x = \mu_1 - \mu_k$ becomes $i + 2m = \mu_1 + \mu_k$ and obtain 
\begin{equation}
\label{eq112}
\sum_{\substack{i+2m = \mu_1+\mu_k \\ m\geq \mu_k+1, i\geq 0}}m{\mu}_kQ_{g,n-1}(i,\mmu_{X\setminus \{1,k\}})+\sum_{\substack{i+2x = \mu_1-\mu_k \\ x\geq 1, i\geq 0}}x{\mu}_kQ_{g,n-1}(i,\mmu_{X\setminus \{1,k\}}).
\end{equation}
\item \label{case1c} There are $x\geq 1$ bigons parallel to $\bar{\alpha} F_1 \alpha$. This is same as the previous case with $F_1$ and $F_k$ interchanged. 
The map $D \mapsto D'$ is $\mu_1\mu_k$-to-$1$, since the bigons now have $\mu_1$ vertices on $F_1$. 
Therefore we have the contribution:
\[
\sum_{\substack{i+2x = \mu_k-\mu_1 \\ x\geq 1, i\geq 0}}\mu_1 {\mu}_k Q_{g,n-1}(i,\mmu_{X\setminus \{1,k\}}).
\]
Writing $\mu_1$ as $(x + \mu_1) - x$ and setting $m = x + \mu_1$, we note that $i+2x = \mu_k - \mu_1$ becomes $i+2m = \mu_1 + \mu_k$, and obtain
\begin{equation}
\label{eq113}
\sum_{\substack{i+2m = \mu_1+\mu_k \\ m\geq \mu_1+1, i\geq 0}}m{\mu}_kQ_{g,n-1}(i,\mmu_{X\setminus \{1,k\}})-\sum_{\substack{i+2x = \mu_k-\mu_1 \\ x\geq 1, i\geq 0}}x{\mu}_kQ_{g,n-1}(i,\mmu_{X\setminus \{1,k\}}).
\end{equation}
\end{enumerate}
Observe that the index set $\{i+2m = \mu_1+\mu_k, m\geq 1, i\geq 0\}$ is the disjoint union of index sets
$\{i+2m = \mu_1+\mu_k, 1\leq m \leq \min(\mu_1,\mu_k), i\geq 0\}$, $\{i+2m = \mu_1+\mu_k,  m \geq \mu_k+1, i\geq 0\}$,
and $\{i+2m = \mu_1+\mu_k,  m \geq \mu_i+1, i\geq 0\}$. 
(If $m \geq \mu_k + 1$ then $\mu_1 + \mu_k = i + 2m \geq 2\mu_k + 2$, hence $\mu_1 \geq \mu_k + 2$; similarly if $m \geq \mu_1 + 1$ then $\mu_k \geq \mu_1 + 2$. So the second and third sets are disjoint.)

Dropping the subscript $g,n-1$ from $Q$ and $X \setminus \{1,k\}$ from $\mmu$ for convenience, we find the sum of \eqref{eq111}, \eqref{eq112}, \eqref{eq113} is
\begin{equation}
\label{eq21}
\sum_{\substack{i+2m = \mu_1+\mu_k \\ m\geq 1, i\geq 0}}m {\mu}_k Q (i,\mmu)+
\sum_{\substack{i+2x = \mu_1-\mu_k \\ x\geq 1, i\geq 0}}x {\mu}_k Q (i,\mmu) 
-\sum_{\substack{i+2x = \mu_k-\mu_1 \\ x\geq 1, i\geq 0}}x {\mu}_k Q (i,\mmu).
\end{equation}

\item There is one totally boundary parallel triangle and $m-1$ bigons.  
\begin{enumerate}
\item \label{case2a} The triangle has two edges parallel to $\alpha$ and the third edge parallel to $\alpha F_k \bar{\alpha}$. 
See figure \ref{fig:case2a}.
The configuration of bigons and triangle is very similar to that of case (B)(1)(b), the only difference is
the innermost bigon parallel to $\alpha F_k \bar{\alpha}$ now becomes the totally boundary parallel triangle.
There are $x-1$ bigons parallel to $\alpha F_k \bar{\alpha}$, $1$ totally boundary parallel triangle, and
$\mu_k-1$ bigons parallel to $\alpha$. An analogous calculation shows we have the contribution
\begin{align}\label{eq22}
\sum_{\substack{i+2m+1 = \mu_1+\mu_k \\ m\geq \mu_k, i\geq 0}}m {\mu}_kQ_{g,n-1}(i,\mmu_{X\setminus \{1,k\}})+
\sum_{\substack{i+2x-1 = \mu_1-\mu_k \\ x\geq 1, i\geq 0}}x {\mu}_kQ_{g,n-1}(i,\mmu_{X\setminus \{1,k\}}).
\end{align}

\begin{figure}
\begin{center}
\includegraphics[scale=0.65]{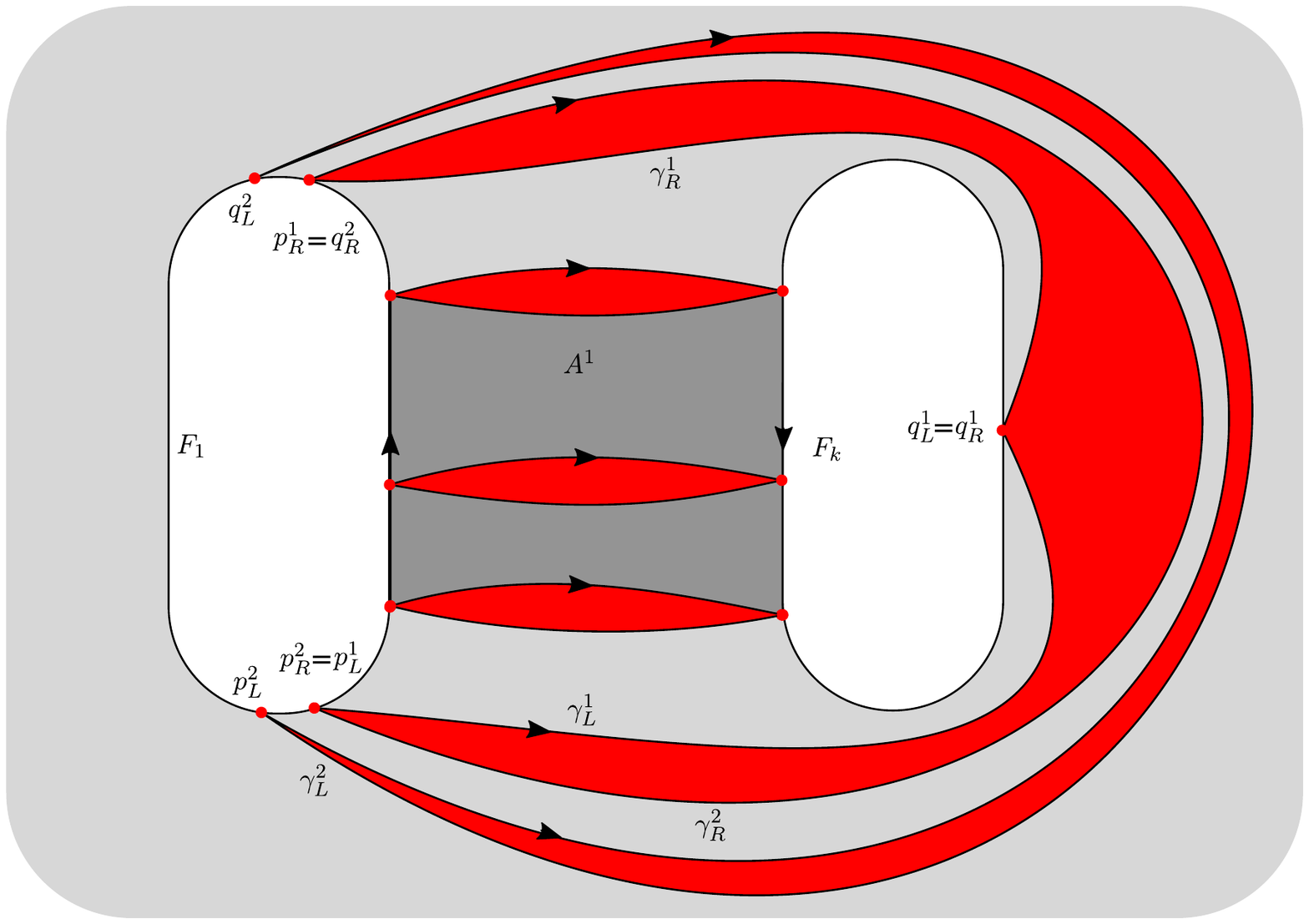}
\caption{Configuration of polygons in case (B)(2)(a).}
\label{fig:case2a}
\end{center}
\end{figure}

\item The triangle has two edges parallel to $\alpha$ and the third edge parallel to $\bar{\alpha} F_1 \alpha$. 
This is very similar to case (B)(1)(c). An analogous calculation shows we have the contribution
\begin{align}\label{eq23}
\sum_{\substack{i+2m+1 = \mu_1+\mu_k \\ m\geq \mu_1, i\geq 0}}m {\mu}_kQ_{g,n-1}(i,\mmu_{X\setminus \{1,k\}})-
\sum_{\substack{i+2x-1 = \mu_k-\mu_1 \\ x\geq 1, i\geq 0}}(x-1) {\mu}_kQ_{g,n-1}(i,\mmu_{X\setminus \{1,k\}}).
\end{align}
\end{enumerate}
\item There are some mixed polygons and a totally boundary parallel triangle. The edge of the triangle not parallel to $\alpha$ is then parallel to either $\alpha F_k \bar{\alpha}$ or $\bar{\alpha} F_1 \alpha$; we consider the two possibilities separately.
\begin{enumerate}
\item \label{case3a} The third edge of the triangle is parallel to $\alpha F_k \bar{\alpha}$. If we view $F_k$ as on the ``inside" of
an edge parallel to $\alpha F_k \bar{\alpha}$, it is easy to see that only the ``outermost" edge
, $\gamma^2_L$ on the minimal strip $A^2$, can be an edge of a mixed polygon. Hence there
is only one mixed polygon, an it is on the outside of $\gamma^2_L$. On the inside of $\gamma^2_L$ we have exactly the 
same configuration of totally boundary parallel polygons as Case (B)(2)(a) and figure \ref{fig:case2a}. There are $\mu_k-1$ bigons parallel to $\alpha$.
Let there be $x-1$ bigons parallel to $\alpha F_k \bar{\alpha}$, and $i$ vertices on $F_1$ outside $\gamma^2_L$.
Then $\mu_1=i+2x+\mu_k+1$ and $m=x+\mu_k$. We obtain a pruned polygon diagram $D'$ by removing all totally 
boundary parallel bigons and triangle, cutting $S$ along $\gamma^2_L$ and collapsing $\gamma^2_L$ into a new vertex on the 
new boundary component $F'_1$ of $S'$, which we set to be the decorated marked point $\mathbf{m}'_1$. Consider the
possible locations of $\mathbf{m}_1$. It can be a vertex on $F_1$ of any of the $[(x-1)+(\mu_k-1)]$ bigons, of which there are $2(x-1) + (\mu_k - 1)$. It can be either of the two vertices of
the triangle on $F_1$. Or it could be the vertex $p^2_L$, but not $q^2_L$, once again due to $\gamma$ being an outgoing edge from
$\mathbf{m}_1$. (If $q^2_L$ is $\mathbf{m}_1$, then $\gamma$ is $\gamma^2_L$. If $\gamma^2_L$ is outgoing, then the polygon 
containing $\gamma^2_L$ is on the inside of $\gamma^2_L$, making it totally boundary parallel, a contradiction.) Hence the
multiplicity of the map $D \mapsto D'$ is $(2(x-1)+(\mu_k-1)+2+1)\mu_k=(2x+\mu_k)\mu_k$. An analogous calculation shows we 
have the contribution
\begin{align}\label{eq24}
\sum_{\substack{i+2m+1 = \mu_1+\mu_k \\ m\geq \mu_k+1, i\geq 0}}m {\mu}_kQ_{g,n-1}(i+1,\mmu_{X\setminus \{1,k\}})+
\sum_{\substack{i+2x+1 = \mu_1-\mu_k \\ x\geq 1, i\geq 0}}x {\mu}_kQ_{g,n-1}(i+1,\mmu_{X\setminus \{1,k\}}).
\end{align}
\item \label{case3b} The third edge of the triangle is parallel to $\bar{\alpha} F_1 \alpha$. This is the same as the previous
case with $F_1$ and $F_k$ interchanged. The map $D \mapsto D'$ is $\mu_1\mu_k$-to-1. An analogous calculation shows we 
have the contribution.
\begin{align}\label{eq25}
\sum_{\substack{i+2m+1 = \mu_1+\mu_k \\ m\geq \mu_1+1, i\geq 0}}m {\mu}_kQ_{g,n-1}(i+1,\mmu_{X\setminus \{1,k\}})-
\sum_{\substack{i+2x+1 = \mu_k-\mu_1 \\ x\geq 1, i\geq 0}}x {\mu}_kQ_{g,n-1}(i+1,\mmu_{X\setminus \{1,k\}}).
\end{align}
\end{enumerate}
\item There are some mixed polygons but no totally boundary parallel triangle. We now split into cases accordingly as there are edges parallel to $\alpha F_k \bar{\alpha}$ or $\bar{\alpha} F_a \alpha$ or not. There cannot be edges parallel to both, so we have 3 sub-cases.
\begin{enumerate}
\item There are no edges parallel to $\alpha F_k \bar{\alpha}$ or $\bar{\alpha} F_1 \alpha$. 
Consider the minimal strip $A^1$ containing all edges parallel to $\alpha$. We now consider the leftmost and rightmost edges of this strip $\gamma^1_L$ and $\gamma^1_R$, and to what extent they coincide. They may (i) be the same edge; or (ii) they may share both endpoints but be distinct edges; or they may share a vertex on (iii) $F_k$ or (iv) $F_1$ only; or they may be disjoint. When they are disjoint, (v) $\gamma_L^1$ or (vi) $\gamma_R^1$ or (vii) both may belong to mixed polygons. This leads to the 7 sub-cases below.
\begin{enumerate}[label=(\roman*)]
\item \label{casei} $\gamma^1_L = \gamma^1_R = \gamma$. Then there are no other edges parallel to $\gamma$ and thus no
bigons. Since $\gamma$ is
an outgoing edge by assumption, it bounds a mixed polygon to the left. This configuration will be covered in 
Case (B)(4)(a)(v) and we do not include the contribution here.
\item $\gamma^1_L$ and $\gamma^1_R$ are distinct edges with the same endpoints. Then $\gamma^1_L$ and $\gamma^1_R$
bound the bigon $A^1$ and there are no other edges parallel to $\gamma$. This means there are no 
mixed polygons, contrary to assumption. Therefore the contribution vanishes in this case.
\item $\gamma^1_L$ and $\gamma^1_R$ share a common vertex $q^1$ on $F_k$ but not on $F_1$.
See figure \ref{case4a3}.
 Consider the boundary of $A^1$
on $F_k$, $[q^1_R,q^1_L]$. This interval could either be a single point $q^1$, or the entire boundary $F_k$. If it is 
a single point, then the polygon containing $\gamma^1_L$ and $\gamma^1_R$ has to be inside $A^1$, so the diagonal joining
$p^1_L$ and $p^1_R$ is boundary parallel, contradicting the assumption of a pruned diagram.
In the case $[q^1_R,q^1_L]$ is all of $F_k$, $\gamma^1_L$ and $\gamma^1_R$
belong to a single ``outermost" mixed polygon, and there are $m-1$ bigons
between $F_1$ and $F_k$. Let $i\geq 0$ be the number of remaining vertices on $F_1$ outside $A^1$. Then 
$i+\mu_k+1=\mu_1$ and we also have $m=\mu_k$. We obtain a pruned polygon diagram by removing all $m-1$
bigons, cutting along the concatenated edge $\gamma^1_L\bar{\gamma}^1_R$ and collapsing $\gamma^1_L\bar{\gamma}^1_R$
into a new vertex. The multiplicity of the map $D \mapsto D'$ is $m\mu_k$, as $\mathbf{m}_1$ can be a vertex of the
$m-1$ bigons or $p^1_L$. Therefore we have the contribution
\begin{align}\label{eq26}
\sum_{\substack{i+2m+1 = \mu_1+\mu_k \\ m= \mu_k, i\geq 0}}m {\mu}_kQ_{g,n-1}(i+1,\mmu_{X\setminus \{1,k\}}).
\end{align}

\begin{figure}
\begin{center}
\includegraphics[scale=0.65]{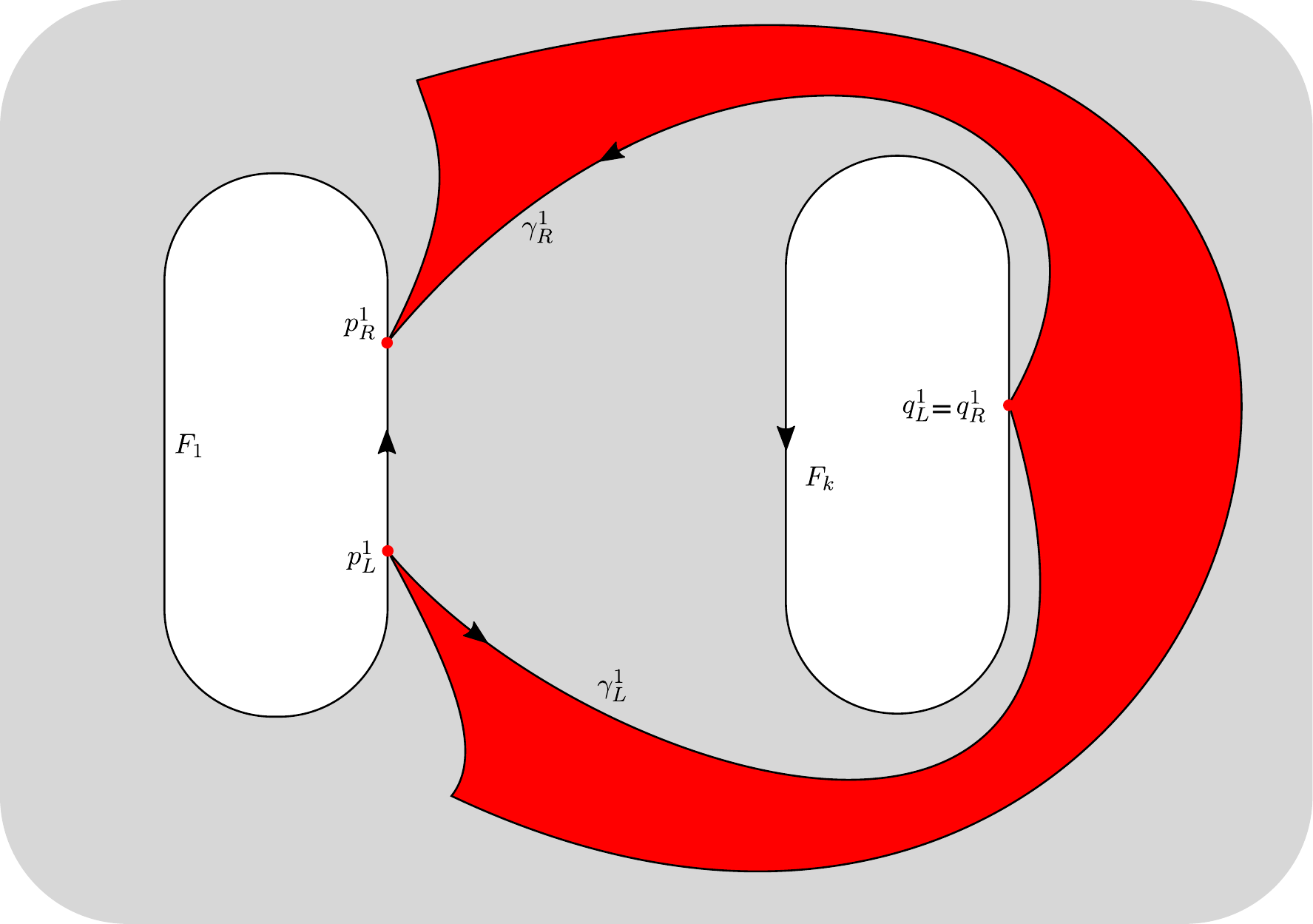}
\caption{Configuration of polygons in case (B)(4)(a)(iii).}
\label{case4a3}
\end{center}
\end{figure}

\item $\gamma^1_L$ and $\gamma^1_R$ share a common vertex $p^1$ on $F_1$ but not on $F_k$. This is the same as the previous
case with $F_1$ and $F_k$ interchanged. The map $D \mapsto D'$ is $\mu_1\mu_k$-to-1. An analogous calculation shows we 
have the contribution
\begin{align}\label{eq27}
\sum_{\substack{i+2m+1 = \mu_1+\mu_k \\ m= \mu_1, i\geq 0}}m {\mu}_kQ_{g,n-1}(i+1,\mmu_{X\setminus \{1,k\}}).
\end{align}
\item \label{caseiii} $\gamma^1_L$ and $\gamma^1_R$ do not share any vertex, and $\gamma^1_L$ belongs to a mixed polygon but
$\gamma^1_R$ does not. There are $m-1\geq 1$ bigons parallel to $\alpha$. Let $i=\mu_i+\mu_k-2m$ be the total number
of remaining vertices on $F_1$ and $F_k$ outside $A^1$. We obtain a pruned polygon diagram $D'$ by removing all $m-1$ bigons,
cutting along $\gamma^1_L$ and collapsing $\gamma^1_L$ into a new vertex. The map $D \mapsto D'$ is $m\mu_k$-to-1. Note that if
we allow $m=1$, this exactly covers the configuration in case (B)(4)(a)(i). 
Therefore we have the contribution
\begin{align}\label{eq28}
\sum_{\substack{i+2m = \mu_1+\mu_k \\ 1\leq m\leq \min(\mu_1,\mu_k), i\geq 0}}m {\mu}_k Q_{g,n-1}(i+1,\mmu_{X\setminus \{1,k\}}).
\end{align}
\item $\gamma^1_L$ and $\gamma^1_R$ do not share any vertex, and $\gamma^1_R$ belongs to a mixed polygon but
$\gamma^1_L$ does not. This is almost exactly the same as the previous case, except $\gamma^1_R$ bounds a mixed polygon
to the right, so it cannot be $\gamma$. It follows that $\mathbf{m}_1$ cannot be $p^2_R$ and the map $D \mapsto D'$ is $(m-1)\mu_k$-to-1.
Therefore we have the contribution:
\begin{align}\label{eq29}
\sum_{\substack{i+2m = \mu_1+\mu_k \\ 1\leq m\leq \min(\mu_1,\mu_k), i\geq 0}}(m-1) {\mu}_kQ_{g,n-1}(i+1,\mmu_{X\setminus \{1,k\}})
\end{align}
Note that we allow $m=1$ in the summation index because the summand vanishes for $m=1$ anyway.
\item $\gamma^1_L$ and $\gamma^1_R$ do not share any vertex, and both belong to mixed polygons (possibly the same one). 
Since there could be $1$ or $2$ mixed polygons, we instead define $m\geq 2$ to be $2$ plus the number of bigons in $A^1$. 
We obtain a pruned polygon diagram $D'$ by removing all $m-2$ bigons, cutting the strip $A^1$ from $S$ along $\gamma^1_L$ 
and $\gamma^1_R$, and collapsing $\gamma^1_L$ and $\gamma^1_R$ into two new vertices. Set the decorated marked point
to be the new vertex from collapsing $\gamma^1_L$. Again since $\gamma$ cannot be $\gamma^1_R$, the map $D \mapsto D'$ is 
$(m-1)\mu_k$-to-1. Therefore we have the contribution (again we trivially include $m=1$ in the summation index)
\begin{align}\label{eq210}
\sum_{\substack{i+2m = \mu_1+\mu_k \\ 1\leq m\leq \min(\mu_1,\mu_k), i\geq 0}}(m-1) {\mu}_kQ_{g,n-1}(i+2,\mmu_{X\setminus \{1,k\}}).
\end{align}
\end{enumerate}
\item There are some edges parallel to $\alpha F_k \bar{\alpha}$. This is the same configuration as case (B)(3)(a), 
just without the
single totally boundary parallel triangle. An analogous calculation shows we have the contribution
\begin{align}\label{eq211}
\sum_{\substack{i+2m = \mu_1+\mu_k \\ m\geq \mu_k+1, i\geq 0}}m {\mu}_kQ_{g,n-1}(i+1,\mmu_{X\setminus \{1,k\}})+
\sum_{\substack{i+2x+2 = \mu_1-\mu_k \\ x\geq 0, i\geq 0}}x {\mu}_kQ_{g,n-1}(i+1,\mmu_{X\setminus \{1,k\}}).
\end{align}
\item There are some edges parallel to $\bar{\alpha} F_1 \alpha$. This is the same configuration as case (B)(3)(b), 
just without the
single totally boundary parallel triangle. An analogous calculation shows we have the contribution
\begin{align}\label{eq212}
\sum_{\substack{i+2m = \mu_1+\mu_k \\ m\geq \mu_1+1, i\geq 0}}m {\mu}_kQ_{g,n-1}(i+1,\mmu_{X\setminus \{1,k\}})-
\sum_{\substack{i+2x+2 = \mu_k-\mu_1 \\ x\geq 0, i\geq 0}}(x+1) {\mu}_kQ_{g,n-1}(i+1,\mmu_{X\setminus \{1,k\}}).
\end{align}
\end{enumerate}
\end{enumerate}
We have exhausted all possibilities in case (B). The total contribution is the sum of all the expressions \eqref{eq21}--\eqref{eq212}, which we now sum. We drop subscripts $g,n-1$ from $Q$ and $X \setminus \{1,k\}$ from $\mmu$ for convenience.

We first calculate the sum of terms with summation over $m$. 
The $m$-summation terms in \eqref{eq24} and \eqref{eq26}, \eqref{eq25} and \eqref{eq27} combine to give
\begin{align}\label{eq31}
\sum_{\substack{i+2m+1 = \mu_1+\mu_k \\ m\geq \mu_k, i\geq 0}}m & {\mu}_k Q (i+1,\mmu)+
\sum_{\substack{i+2m+1 = \mu_1+\mu_k \\ m\geq \mu_1, i\geq 0}}m {\mu}_k Q(i+1,\mmu)\nonumber \\
=&\sum_{\substack{i+2m = \mu_1+\mu_k \\ m\geq \mu_k, i\geq 1}}m {\mu}_k Q(i,\mmu)+
\sum_{\substack{i+2m = \mu_1+\mu_k \\ m\geq \mu_1, i\geq 1}}m {\mu}_k Q(i,\mmu).
\end{align}
We rewrite the $m$-summation term in \eqref{eq210}, using the substitution $(m',i')=(m-1,i+2)$, and then adding a vacuous summation index $i=1$,
since  $1+2m = \mu_1+\mu_k$ and $m \leq \min(\mu_1,\mu_k)-1$ cannot hold simultaneously. We obtain
\begin{equation}\label{eq32}
\sum_{\substack{i+2m = \mu_1+\mu_k \\ 0\leq m \leq \min(\mu_1,\mu_k)-1, i\geq 1}}m {\mu}_k Q (i,\mmu ).
\end{equation}
Since the index set $\{i+2m = \mu_1+\mu_k, m\geq 0, i\geq 1\}$ is the disjoint union of index sets
$\{i+2m = \mu_1+\mu_k, 0\leq m \leq \min(\mu_1,\mu_k)-1, i\geq 1\}$, $\{i+2m = \mu_1+\mu_k,  m \geq \mu_k, i\geq 1\}$,
and $\{i+2m = \mu_1+\mu_k,  m \geq \mu_i, i\geq 1\}$,
\eqref{eq31} 
and \eqref{eq32} 
 sum to
\begin{align}\label{eq40}
\sum_{\substack{i+2m = \mu_1+\mu_k \\ m\geq 0, i\geq 1}}m {\mu}_k Q (i,\mmu) = \sum_{\substack{i+2m = \mu_1+\mu_k \\ m\geq 1, i\geq 1}}m {\mu}_k Q (i,\mmu),
\end{align}
which is the sum of all $m$-summation terms in \eqref{eq24}, \eqref{eq25}, \eqref{eq26}, \eqref{eq27} and \eqref{eq210}.

The $m$-summation terms in \eqref{eq21} and \eqref{eq40}
combine to give
\begin{align}\label{eq401}
\left( \sum_{\substack{i+2m = \mu_1+\mu_k \\ m\geq 1, i\geq 0}} +
\sum_{\substack{i+2m = \mu_1+\mu_k \\ m\geq 1, i\geq 1}} \right) m{\mu}_k Q (i,\mmu) 
&= \sum_{\substack{i+2m = \mu_1+\mu_k \\ m\geq 1, i\geq 1}}2m {\mu}_k Q(i,\mmu) +
\sum_{\substack{i+2m = \mu_1+\mu_k \\ m\geq 1, i=0}}m {\mu}_k Q(i,\mmu) \nonumber \\
=&\sum_{\substack{i+2m = \mu_1+\mu_k \\ m\geq 0, i\geq 1}}2m {\mu}_k Q(i,\mmu) +
 \frac{\widetilde{(\mu_1+\mu_k)}}{2}{\mu}_k Q(0,\mmu),
\end{align}
where we use the $\widetilde{\mu}$ notation of definition \ref{tilde_notation} in the final term.
This is the sum of all $m$-summation terms in \eqref{eq21}, \eqref{eq24}, \eqref{eq25}, \eqref{eq26}, \eqref{eq27}, \eqref{eq210}.

We next rewrite the $m$-summation terms from \eqref{eq28} and \eqref{eq29} with the substitution $(m',i') = (m-1,i+1)$ to obtain
\begin{align}\label{int1}
\sum_{\substack{i+2m = \mu_1+\mu_k \\ 1\leq m\leq \min(\mu_1,\mu_k), i\geq 0}} & m {\mu}_kQ(i+1,\mmu)
+\sum_{\substack{i+2m = \mu_1+\mu_k \\ 1\leq m\leq \min(\mu_1,\mu_k), i\geq 0}}(m-1) {\mu}_kQ(i+1,\mmu)\nonumber \\
&= \sum_{\substack{i+2m+1 = \mu_1+\mu_k \\ 0\leq m\leq \min(\mu_1,\mu_k)-1, i\geq 1}} (2m+1) {\mu}_kQ(i,\mmu), 
\end{align}
and similarly with \eqref{eq211}, and \eqref{eq212} to obtain 
\begin{align}\label{int2}
\sum_{\substack{i+2m = \mu_1+\mu_k \\ m\geq \mu_k+1, i\geq 0}} &m {\mu}_kQ(i+1,\mmu)
+\sum_{\substack{i+2m = \mu_1+\mu_k \\ m\geq \mu_1+1, i\geq 0}}m {\mu}_kQ(i+1,\mmu) \nonumber \\
&=
 \left( \sum_{\substack{i+2m+1 = \mu_1+\mu_k \\ m\geq \mu_k, i\geq 1}} + \sum_{\substack{i+2m+1= \mu_1+\mu_k \\ m\geq \mu_1, i\geq 1}} \right) (m+1) {\mu}_kQ(i,\mmu).
\end{align}
Now combining the $m$-summation terms in \eqref{eq22}, \eqref{eq23}, \eqref{int1}, \eqref{int2}
we obtain
\begin{align}\label{eq41}
\sum_{\substack{i+2m+1 = \mu_1+\mu_k \\ m\geq \mu_k, i\geq 0}} &m {\mu}_kQ(i,\mmu)
+\sum_{\substack{i+2m+1 = \mu_1+\mu_k \\ m\geq \mu_1, i\geq 0}}m {\mu}_kQ(i,\mmu) 
+ \sum_{\substack{i+2m+1 = \mu_1+\mu_k \\ 0\leq m\leq \min(\mu_1,\mu_k)-1, i\geq 1}} (2m+1) {\mu}_kQ(i,\mmu) \nonumber \\
&+ \left( \sum_{\substack{i+2m+1 = \mu_1+\mu_k \\ m\geq \mu_k, i\geq 1}} + \sum_{\substack{i+2m+1= \mu_1+\mu_k \\ m\geq \mu_1, i\geq 1}} \right) (m+1) {\mu}_kQ(i,\mmu) \nonumber \\
=&\sum_{\substack{i+2m+1 = \mu_1+\mu_k \\ m\geq 0, i\geq 1}}(2m+1) {\mu}_kQ(i,\mmu) + 
\left[\sum_{\substack{2m+1 = \mu_1+\mu_k \\ m\geq \mu_k}}
+\sum_{\substack{2m+1 = \mu_1+\mu_k \\ m\geq \mu_1}}\right] m {\mu}_kQ(0,\mmu) \nonumber \\
=&\sum_{\substack{i+2m+1 = \mu_1+\mu_k \\ m\geq 0, i\geq 1}}(2m+1) {\mu}_kQ(i,\mmu)+
\frac{\widetilde{(\mu_1+\mu_k-1)}}{2}{\mu}_kQ(0,\mmu).
\end{align}
This is the sum of all $m$-summation terms in \eqref{eq22}, \eqref{eq23}, \eqref{eq28}, \eqref{eq29}, \eqref{eq211}, and \eqref{eq212}.

Adding \eqref{eq401} and \eqref{eq41}, we have the total of all $m$-summation terms:
\begin{align}\label{eq50}
\sum_{\substack{i+m = \mu_1+\mu_k \\ i\geq 1, m\geq 0}}m {\mu}_kQ (i,\mmu) + \frac{\widetilde{(\mu_1+\mu_k)}}{2}{\mu}_k Q(0,\mmu) +
\frac{\widetilde{(\mu_1+\mu_k-1)}}{2}{\mu}_k Q(0,\mmu)
\end{align}
Now we sum the terms with summation over $x$. These arise in expressions \eqref{eq21}, \eqref{eq22}, \eqref{eq23}, \eqref{eq24}, \eqref{eq25}, \eqref{eq211} and \eqref{eq212}.
The total is 
\begin{align}\label{eq51}
&\sum_{\substack{i+2x = \mu_1-\mu_k \\ x\geq 1, i\geq 0}}x {\mu}_kQ(i,\mmu) 
-\sum_{\substack{i+2x = \mu_k-\mu_1 \\ x\geq 1, i\geq 0}}x {\mu}_kQ(i,\mmu) 
+\sum_{\substack{i+2x-1 = \mu_1-\mu_k \\ x\geq 1, i\geq 0}}x {\mu}_kQ(i,\mmu) \nonumber \\
&-\sum_{\substack{i+2x-1 = \mu_k-\mu_1 \\ x\geq 1, i\geq 0}}(x-1) {\mu}_kQ(i,\mmu) 
+\sum_{\substack{i+2x+1 = \mu_1-\mu_k \\ x\geq 1, i\geq 0}}x {\mu}_kQ(i+1,\mmu) - 
\sum_{\substack{i+2x+1 = \mu_k-\mu_1 \\ x\geq 1, i\geq 0}}x {\mu}_kQ(i+1,\mmu) \nonumber \\
&+\sum_{\substack{i+2x+2 = \mu_1-\mu_k \\ x\geq 0, i\geq 0}}x {\mu}_kQ(i+1,\mmu) -
\sum_{\substack{i+2x+2 = \mu_k-\mu_1 \\ x\geq 0, i\geq 0}}(x+1) {\mu}_kQ(i+1,\mmu) \nonumber \\
=&\sum_{\substack{i+2x = \mu_1-\mu_k \\ x\geq 0, i\geq 0}}x {\mu}_kQ(i,\mmu) 
-\sum_{\substack{i+2x = \mu_k-\mu_1 \\ x\geq 0, i\geq 0}}x {\mu}_kQ(i,\mmu) 
+\sum_{\substack{i+2x+1 = \mu_1-\mu_k \\ x\geq 0, i\geq 0}}(x+1) {\mu}_kQ(i,\mmu) \nonumber \\
&-\sum_{\substack{i+2x+1 = \mu_k-\mu_1 \\ x\geq 0, i\geq 0}}x {\mu}_kQ(i,\mmu) 
+\sum_{\substack{i+2x = \mu_1-\mu_k \\ x\geq 0, i\geq 1}}x {\mu}_kQ(i,\mmu) - 
\sum_{\substack{i+2x = \mu_k-\mu_1 \\ x\geq 0, i\geq 1}}x {\mu}_kQ(i,\mmu) \nonumber \\
&+\sum_{\substack{i+2x+1 = \mu_1-\mu_k \\ x\geq 0, i\geq 1}}x {\mu}_kQ(i,\mmu) -
\sum_{\substack{i+2x+1 = \mu_k-\mu_1 \\ x\geq 0, i\geq 1}}(x+1) {\mu}_kQ(i,\mmu) \nonumber \\
=& \sum_{\substack{i+2x = \mu_1-\mu_k \\ x\geq 0, i\geq 1}}2x {\mu}_kQ(i,\mmu) + \frac{\widetilde{(\mu_1-\mu_k)}}{2} {\mu}_kQ(0,\mmu) \nonumber \\
&+\sum_{\substack{i+2x+1 = \mu_1-\mu_k \\ x\geq 0, i\geq 1}}(2x+1) {\mu}_kQ(i,\mmu) + \frac{\widetilde{(\mu_1-\mu_k+1)}}{2} {\mu}_kQ(0,\mmu) \nonumber \nonumber \\
&- \sum_{\substack{i+2x = \mu_k-\mu_1 \\ x\geq 0, i\geq 1}}2x {\mu}_kQ(i,\mmu) -  \frac{\widetilde{(\mu_k-\mu_1)}}{2} {\mu}_kQ(0,\mmu) \nonumber \\
&- \sum_{\substack{i+2x+1 = \mu_k-\mu_1 \\ x\geq 0, i\geq 1}}(2x+1) {\mu}_kQ(i,\mmu) -  \frac{\widetilde{(\mu_k-\mu_1-1)}}{2} {\mu}_kQ(0,\mmu) \nonumber \\
=&\sum_{\substack{i+x = \mu_1-\mu_k \\ x\geq 0, i\geq 1}}x {\mu}_kQ(i,\mmu) - \sum_{\substack{i+x = \mu_k-\mu_1 \\ x\geq 0, i\geq 1}}x {\mu}_kQ(i,\mmu) \nonumber \\
&+\left(\frac{\widetilde{(\mu_1-\mu_k)}}{2}+\frac{\widetilde{(\mu_1-\mu_k+1)}}{2}- \frac{\widetilde{(\mu_k-\mu_1)}}{2} -\frac{\widetilde{(\mu_k-\mu_1-1)}}{2}\right) {\mu}_kQ(0,\mmu)
\end{align}

It is not hard to verify that for $\mu_1,\mu_k\geq 1$,
\begin{align*}
\mu_1 = 
\frac{\widetilde{(\mu_1+\mu_k)}}{2}+\frac{\widetilde{(\mu_1+\mu_k-1)}}{2} +
\frac{\widetilde{(\mu_1-\mu_k)}}{2}+\frac{\widetilde{(\mu_1-\mu_k+1)}}{2}- \frac{\widetilde{(\mu_k-\mu_1)}}{2} -\frac{\widetilde{(\mu_k-\mu_1-1)}}{2}
\end{align*}
Hence combining \eqref{eq50} and \eqref{eq51} we have the second line of \eqref{qrecursion-eq}.

If $\mu_k = 0$, then there are only two possible configuration of partially boundary parallel polygons. 
Either they form $m$ bigons parallel to $\alpha F_k \bar{\alpha}$, or they form $m-1$ bigons and the outermost edge is parallel to $\alpha F_k \bar{\alpha}$
belongs to a mixed polygon. These two configurations respectively contribute the two terms of 
\[
\sum_{\substack{i+2m=\mu_1 \\ i\geq 0, m\geq 1}}2m Q_{g,n-1}(i,\mmu_{X\setminus \{1,k\}})
+
\sum_{\substack{i+2m=\mu_1 \\ i\geq 0, m\geq 1}}(2m-1) Q_{g,n-1}(i+1,\mmu_{X\setminus \{1,k\}}).
\]
Adding a zero term to the first sum and reparametrising the second, this expression becomes
\begin{align*}
\sum_{\substack{i+2m=\mu_1 \\ i\geq 0, m\geq 0}}& 2m Q_{g,n-1}(i,\mmu_{X\setminus \{1,k\}})+\sum_{\substack{i+2m+1=\mu_1 \\ i\geq 1, m\geq 0}}(2m+1) Q_{g,n-1}(i,\mmu_{X\setminus \{1,k\}}) \nonumber \\
=&\sum_{\substack{i+m = \mu_1 \\ i\geq 1, m\geq 0}}m Q_{g,n-1}(i,\mmu_{X\setminus \{1,k\}})+ \widetilde{\mu}_1Q_{g,n-1}(0,\mmu_{X\setminus \{1,k\}})
\end{align*}
This gives the third line of \eqref{qrecursion-eq}.
\item[(C) $\gamma$ has both ends on $S_1$, is separating, and does not cut off an annulus.]~

The configurations in this case are almost identical to those in case $\mathbf{(A)}$, where $\gamma$ is non-separating. The calculation is
formally identical, we simply substitute $Q_{g_1,|I|+1}(\triangle,\mmu_{I})Q_{g_2,|J|+1}(\square,\mmu_{J})$
in place of $Q_{g-1,n+1}(\triangle,\square,\mmu_{X\setminus \{1\}})$ everywhere. We obtain the last line of \eqref{qrecursion-eq}.
\end{description}
\end{proof}

\subsection{Counts for punctured tori}

With the recursion \eqref{qrecursion-eq} of theorem \ref{qrecursion} in hand, we now obtain the count of pruned polygon diagrams on punctured tori, using the established count for annuli in proposition \ref{basecases}. Then, using proposition \ref{PQ}, we obtain the count of general polygon diagrams.

\begin{proposition}\label{basecase2}
\begin{align*}
Q_{1,1}(\mu_1) &=  
\begin{cases}
\frac{\mu_1^3-\mu_1}{24}, & \mu_1 > 0 \text{ odd} \\
\frac{\mu_1^3+8\mu_1}{24}, & \mu_1 > 0 \text{ even} \\
1, & \mu_1=0
\end{cases} 
\end{align*}
\end{proposition}
\begin{proof}
For $(g,n)=(1,1)$ the recursion \eqref{qrecursion-eq} reduces to 
\begin{align*}
Q_{1,1} (\mu_1) = \sum_{\substack{i+j+m = \mu_1 \\ i\geq 1, j,m\geq 0}}m Q_{0,2}(i,j) + \frac{\widetilde{\mu}_1}{2}Q_{0,2}(0,0)
\end{align*}
By Proposition \ref{basecases}, $Q_{0,2}(i,j) = \overline{i} \delta_{i,j}$. If $\mu_1 > 0$ is odd, then we have
\[
Q_{1,1} (\mu_1) = \sum_{\substack{2i+m = \mu_1 \\ i,m\geq 1}}m i = \frac{1}{2}\sum_{\substack{0\leq m \leq\mu_1-2 \\ m \text{ odd}}}m (\mu_1-m) 
= \frac{\mu_1}{2} \sum_{\substack{0 \leq m \leq \mu_1 - 2 \\ m \text{ odd}}} m - \frac{1}{2} \sum_{\substack{0 \leq m \leq \mu_1 - 2 \\ m \text{ odd}}} m^2.
\]
Lemma \ref{lem-odd-even-power-sums} gives the two sums immediately, and we obtain
\[
Q_{1,1}(\mu_1) = \frac{\mu_1}{2} \frac{(\mu_1 - 1)^2}{4} - \frac{1}{2} \frac{(\mu_1 - 2)(\mu_2 - 1)\mu_2}{6}
= \frac{\mu_1^3 - \mu_1}{24}.
\]
If $\mu_1 > 0$ is even, then similarly we have
\[
Q_{1,1} (\mu_1) 
= \sum_{\substack{2i+m = \mu_1 \\ i,m\geq 1}}m i + \frac{\mu_1}{2} 
= \frac{1}{2}\sum_{\substack{0\leq m \leq\mu_1-2 \\ m \text{ even}}}m (\mu_1-m) + \frac{\mu_1}{2}  
= \frac{\mu_1}{2} \sum_{\substack{0 \leq m \leq \mu_1 - 2 \\ m \text{ even}}} m - \frac{1}{2} \sum_{\substack{0 \leq m \leq \mu_1 - 2 \\ m \text{even}}} m^2 + \frac{\mu_1}{2},
\]
and lemma \ref{lem-odd-even-power-sums} then yields
\[
Q_{1,1} (\mu_1) = \frac{\mu_1}{2} \frac{(\mu_1 - 2)\mu_1}{4}
- \frac{1}{2} \frac{(\mu_1 - 2)(\mu_1 - 1)\mu_1}{6}
+ \frac{\mu_1}{2}
= \frac{\mu_1^3 + 8 \mu_1}{24}.
\]
\end{proof}

\begin{proposition}\label{P11count}
\begin{align*}
P_{1,1}(\mu_1)&=\binom{2\mu-1}{\mu} \frac{1}{2\mu-1} \frac{\mu^3 + 3\mu^2 + 20\mu - 12}{12}
\end{align*}
\end{proposition}
\begin{proof} By Proposition \ref{PQ}, for $\mu_1 > 0$, and then by proposition \ref{basecase2},
\begin{align*}
P_{1,1}(\mu_1)&=\sum_{\nu_1\leq \mu_1, \nu_1 \text{ odd}}Q_{1,1}(\nu_1)\binom{2\mu_1}{\mu_1-\nu_1} + \sum_{\nu_1\leq \mu_1, \nu_1 \text { even}}Q_{1,1}(\nu_1)\binom{2\mu_1}{\mu_1-\nu_1}\\
&=\sum_{\nu_1\leq \mu_1, \nu_1 \text{ odd}}\frac{\nu_1^3-\nu_1}{24}\binom{2\mu_1}{\mu_1-\nu_1} + \sum_{\nu_1\leq \mu_1, \nu_1 \text{ even}}\frac{\nu_1^3+8\nu_1}{24}\binom{2\mu_1}{\mu_1-\nu_1}
\end{align*}
Using the combinatorial identities \eqref{comb_id_oe1}--\eqref{comb_id_3o}, this simplifies to $\binom{2\mu-1}{\mu} \frac{1}{2\mu-1} \frac{\mu^3 + 3\mu^2 + 20\mu - 12}{12}$.
\end{proof}

We have now proved proposition \ref{Pexamples}, with equations \eqref{eqn:P01_formula}--\eqref{P11} proved in the introduction and  propositions \ref{P02prop}, \ref{P03count}, and \ref{P11count} respectively.

\section{Polynomiality}

We now prove theorem \ref{thm:quasipolynomiality}, that $Q_{g,n}(\mu_1, \ldots, \mu_n)$ is an odd quasi-polynomial for $(g,n) \neq (0,1),(0,2)$. 
The proof follows in the same fashion as proposition \ref{basecase2}.


\begin{proof}[Proof of theorem \ref{thm:quasipolynomiality}]
We use induction on the negative Euler characteristic $-\chi=2g-2+n$. When $2g-2+n=-1$, $(g,n) = (0,3)$ or $(1,1)$, theorem holds by propositions \ref{basecases} and \ref{basecase2}. 
Fix the parities/vanishings of $(\mu_1,\ldots,\mu_n)$. We split the right hand side of the recursion equation \eqref{qrecursion-eq} for $Q_{g,n}$ into $9$ partial sums 
depending on the parities/vanishings of $(i,j)$. We will show that each partial sum is a polynomial. Within each partial sum, since 
the parities/vanishings of $(i,j,\mu_1,\ldots,\mu_n)$ are fixed, $Q_{g-1,n+1}$, $Q_{g,n-1}$, $Q_{g_1,|I|+1}$ and $Q_{g_2,|J|+1}$ are polynomials 
by the induction assumption. Split each polynomial into monomials in $(i,j,\mu_1,\ldots,\mu_n)$. 
To show odd quasi-polynomiality it is 
sufficient to show that for $(i,j)$ with fixed parities/vanishings, and for odd positive integers $K$ and $L$, the following statements hold. (The degrees $K$ and $L$ remain odd by assumption.)
\begin{enumerate}
\item $A(\mu_1)=\sum_{\substack{i+j+m = \mu_1 \\ i\geq 1, j,m\geq 0}}mi^{K}j^{L}$ is an odd polynomial in $\mu_1$,
\item $B(\mu_1,\mu_k)=\left( \sum_{\substack{i+m = \mu_1+\mu_k \\ i\geq 1, m\geq 0}}m \mu_k i^{K} +
\widetilde{\sum_{\substack{i+x = \mu_1-\mu_k \\ i\geq 1, x\geq 0}}} x \mu_k i^{K} \right)$ is an odd polynomial in $\mu_1$ and $\mu_k$,
\item $C(\mu_1)=\sum_{\substack{i+m = \mu_1 \\ i\geq 1, m\geq 0}}mi^{K}$ is an odd polynomial in $\mu_1$.
\end{enumerate}
For the first statement, we have
\[
A(\mu_1) =\sum_{\substack{i+j+m = \mu_1 \\ i\geq 1, j,m\geq 0}}mi^{K}j^{L} = \sum_{\substack{i+j+m = \mu_1 \\ i,j,m\geq 1}}m i^{K}j^{L} 
= \sum_{\substack{i+j+m = \mu_1 \\ i,j,m\geq 1, m \text{ even}}}m i^{K}j^{L} + \sum_{\substack{i+j+m = \mu_1 \\ i,j,m\geq 1, m \text{ odd}}}m i^{K}j^{L}
\]
Since $(i,j)$ have fixed parities and $K,L$ are odd, it follows from proposition \ref{lemma-odd-induction} that $A(\mu_1)$ an odd polynomial in $\mu_1$.
A similar argument show $C(\mu_1)$ is an odd polynomial in $\mu_1$. As for $B(\mu_1, \mu_2)$, another application of proposition \ref{lemma-odd-induction} that for some odd polynomial $P(x)$,
polynomial $P(x)$,
\begin{align*}
B(\mu_1,\mu_k)=
 \sum_{\substack{i+m = \mu_1+\mu_k \\ i\geq 1, m\geq 0}}m \mu_k i^{K} +
\widetilde{\sum_{\substack{i+x = \mu_1-\mu_k \\ i\geq 1, x\geq 0}}}x \mu_k i^{K} 
=&
\begin{cases} 
\mu_kP(\mu_1+\mu_k)+\mu_kP(\mu_1-\mu_k),\ \mu_1\geq \mu_k \\
\mu_kP(\mu_1+\mu_k)-\mu_kP(\mu_k-\mu_1),\ \mu_1< \mu_k 
\end{cases} \\
=&\ \ \mu_k[P(\mu_1+\mu_k)+P(\mu_1-\mu_k)]
\end{align*}
That $P$ is odd then implies that $B(\mu_1,\mu_k)$ is odd 
with respect to both $\mu_1$ and $\mu_k$.
\end{proof}

If we keep track of the degrees of the polynomials in Proposition \ref{lemma-odd-induction}, we see from the recursion \eqref{qrecursion-eq} 
only the top degree terms in $Q_{g,n-1}$, $Q_{g_1,|I|+1}$ and $Q_{g_2,|J|+1}$ can contribute to the top degree component of $Q_{g,n}^{(X_e,X_o,X_\emptyset)}$.
Going through each term on the right hand side of \eqref{qrecursion-eq}, it is easy to verify by induction that
\begin{itemize}
\item the degree of $Q_{g,n}^{(X_e,X_o,\emptyset)}$ is $6g-6+3n$ (i.e. when $X_\emptyset = \emptyset$ and all variables $\mu_1, \ldots, \mu_n$ are nonzero), 
\item the degree of $Q_{g,n}^{(X_e,X_o,X_\emptyset)}$ is at most $6g-6+3n-|X_0|$ if $X_0$ is non-empty,
\end{itemize}
Furthermore, since the leading coefficient of the resultant odd polynomial in Proposition \ref{lemma-odd-induction} is independent of parities, it 
again follows by induction that for $\mu_1,\ldots,\mu_n \geq 1$, the top degree component of $Q_{g,n}(\mu_1,\ldots,\mu_n)$ is independent of 
the choice of parities of the $\mu_i$'s.

Let $[Q_{g,n}(\mu_1,\ldots,\mu_n)]^{\mathrm{top}}$ denote this common top degree component of the quasi-polynomial $Q_{g,n}$. 
Then for positive $\mu_i$'s the recursion \eqref{qrecursion-eq} truncates to 

\begin{align}\label{toprecur}
[Q_{g,n}&(\mu_1, \ldots, \mu_n)]^{\mathrm{top}} 
= \left[\sum_{\substack{i+j+m = \mu_1 \\ i,j,m\geq 1}}m [Q_{g-1,n+1}(i,j,\mmu_{X\setminus \{1\}})]^{\mathrm{top}}\right]^{\mathrm{top}}  \nonumber\\  
&+ \left[\sum_{2\leq j\leq n}\left( \sum_{\substack{i+m = \mu_1+\mu_k \\ i,m\geq 1}}m \mu_k [Q_{g,n-1}(i,\mmu_{X\setminus \{1,k\}})]^{\mathrm{top}} +
\widetilde{\sum_{\substack{i+x = \mu_1-\mu_k \\ i,x\geq 1}}}x \mu_k [Q_{g,n-1}(i,\mmu_{X\setminus \{1,k\}})]^{\mathrm{top}}
\right)\right]^{\mathrm{top}} \nonumber\\
&+ \left[\sum_{\substack{g_1+g_2=g \\ I\sqcup J = \{2,\ldots, n\}  \\ \text{No discs or annuli}}}
\left(\sum_{\substack{i+j+m=\mu_1 \\ i,j,m\geq 1}}m [Q_{g_1,|I|+1}(i,\mmu_{I})]^{\mathrm{top}}[Q_{g_2,|J|+1}(j,\mmu_{J})]^{\mathrm{top}}\right) \right]^{\mathrm{top}}
\end{align}

We now compare the pruned polygon diagram counts $Q_{g,n}$ to the non-boundary-parallel (i.e. pruned) arc diagram counts $N_{g,n}$ of \cite{DKM2017}. 
We observe from the following two theorems that $N_{g,n}$ satisfies some initial conditions and recursion similar to those of $Q_{g,n}$.
\begin{proposition}[\cite{DKM2017} prop. 1.5]
\label{prop:Ngn_basecase}
\begin{align*}
N_{0,3}(\mu_1,\mu_2,\mu_3) &=
\begin{cases}
\bar{\mu_1}\bar{\mu_2}\bar{\mu_2}, & \mu_1+\mu_2+\mu_3 \text{ even}\\
0, & \mu_1+\mu_2+\mu_3 \text{ odd}
\end{cases}
\quad \quad \text{and} \quad \quad
N_{1,1}(\mu_1) &=
\begin{cases}
\frac{\mu_1^3+20\mu_1}{48}, & \mu_1 > 0 \text{ even}\\
0, & \mu_1 > 0 \text{ odd}\\
1, & \mu_1 = 0.
\end{cases}
\end{align*}
\end{proposition}
 
\begin{proposition}[\cite{DKM2017} prop. 6.1]\label{initialN}
\label{prop:Ngn_recursion}
For $(g,n) \neq (0,1), (0,2), (0,3)$ and integers $\mu_1 >0$, $\mu_2, \ldots, \mu_n \geq 0$,
\begin{align*}
N_{g,n}(\mu_1, \ldots, \mu_n) &= \sum_{\substack{ i,j,m \geq 0 \\ i+j+m = \mu_1 \\ m \text{ even}}} \frac{m}{2} \; N_{g-1,n+1} (i,j,\mmu_{X\setminus \{1\}}) \\
& + \sum_{\substack{\mu_k>0 \\ 2\leq j\leq n}} \left( \sum_{\substack{i,m \geq 0 \\ i+m = \mu_1 + \mu_k \\ m \text{ even}}} \frac{m}{2} \; \mu_k \; N_{g,n-1} (i,\mmu_{X\setminus \{1,k\}}) + \widetilde{\sum_{\substack{i,m \geq 0 \\ i+m = \mu_1 - \mu_k \\ m \text{ even}}}} \frac{m}{2} \; \mu_k \; N_{g,n-1} (i, \mmu_{X\setminus \{1,k\}}) \right) \\
& + \sum_{\substack{\mu_k=0 \\ 2\leq j\leq n}} \left( \sum_{\substack{i,m \geq 0 \\ i+m = \mu_1 \\ m \text{ even}}} \frac{m}{2} \; N_{g,n-1} (i,\mmu_{X\setminus \{1,k\}})\right) \\
& + \sum_{\substack{g_1 + g_2 = g \\ I \sqcup J = \{2, \ldots, n\} \\ \text{No discs or annuli}}} \sum_{\substack{i,j,m \geq 0 \\ i+j+m = \mu_1 \\ m \text{ even}}} \frac{m}{2} \; N_{g_1, |I|+1} (i, \mmu_I) \; N_{g_2, |J|+1} (j, \mmu_J)
\end{align*}
\end{proposition}

Using the same argument as for $Q_{g,n}$, the first and third authors with Koyama showed that $N_{g,n}$ is an odd quasi-polynomial such that
\begin{itemize}
\item if $\sum_{i=1}^n \mu_i$ is odd, then $N_{g,n}(\mu_1, \ldots, \mu_n) = 0$,
\item if $\sum_{i=1}^n \mu_i$ is even, then the degree of $N_{g,n}^{(X_e,X_o,\emptyset)}(\mu_1, \ldots, \mu_n)$ is $6g-6+3n$ (i.e. when all $\mu_i$ are nonzero),
\item the degree of $N_{g,n}^{(X_e,X_o,X_0)}$ is at most $6g-6+3n-|X_0|$ if $X_0$ is non-empty.
\end{itemize}

Furthermore the leading coefficients of $N_{g,n}$ encode the intersection numbers on the compactified moduli 
space $\overline{\mathcal{M}}_{g,n}$.
\begin{theorem}[\cite{DKM2017} thm. 1.9]\label{Nintersection}
For $(g,n) \neq (0,1)$ or $(0,2)$, and $\mu_1,\ldots,\mu_n\geq 1$ such that $\sum \mu_i$ is even, 
the polynomial $N_{g,n}^{(X_e,X_o,\emptyset)}(\mu_1, \ldots, \mu_n)$ has degree $6g-6+3n$.
The coefficient $c_{d_1,\ldots,d_n}$ of the highest degree monomial $\mu_1^{2d_1+1} \; \cdots \; \mu_n^{2d_n+1}$
is independent of the partition $(X_e,X_o)$, and
$$c_{d_1,\ldots,d_n} = \frac{1}{2^{5g-6+2n}d_1!\cdots d_n!}\int_{\overline{\mathcal{M}}_{g,n}}\psi_1^{d_1}\cdots\psi_n^{d_n}.$$
\end{theorem}

By comparing the recursions on top-degree terms, we show they are equal up to a constant factor.
\begin{proposition}\label{QN} 
For $(g,n) \neq (0,1)$ or $(0,2)$, and $\mu_1,\ldots,\mu_n\geq 1$ such that $\sum \mu_i$ is even,
$$[Q_{g,n}(\mu_1,\ldots,\mu_n)]^{\mathrm{top}} = 2^{4g+2n-5}[N_{g,n}(\mu_1,\ldots,\mu_n)]^{\mathrm{top}}.$$
\end{proposition}
\begin{proof}
The top degree component of $N_{g,n}$ satisfies the recursion
\begin{align}
[N_{g,n}&(\mu_1, \ldots, \mu_n)]^{\mathrm{top}} = \left[\sum_{\substack{ i,j,m \geq 1 \\ i+j+m = \mu_1 \\ m \text{ even}}} \frac{m}{2} \; [N_{g-1,n+1} (i,j,\mmu_{X\setminus \{1\}})]^{\mathrm{top}}\right]^{\mathrm{top}} \nonumber \\
& + \left[\sum_{\substack{\mu_k>0 \\ 2\leq j\leq n}} \left( \sum_{\substack{i,m \geq 1 \\ i+m = \mu_1 + \mu_k \\ m \text{ even}}} \frac{m}{2} \; \mu_k \; [N_{g,n-1} (i,\mmu_{X\setminus \{1,k\}})]^{\mathrm{top}} + \widetilde{\sum_{\substack{i,m \geq 1 \\ i+m = \mu_1 - \mu_k \\ m \text{ even}}}} \frac{m}{2} \; \mu_k \; [N_{g,n-1} (i, \mmu_{X\setminus \{1,k\}})]^{\mathrm{top}} \right)\right]^{\mathrm{top}} \nonumber \\
& + \left[\sum_{\substack{g_1 + g_2 = g \\ I \sqcup J = \{2, \ldots, n\} \\ \text{No discs or annuli}}} \sum_{\substack{i,j,m \geq 1 \\ i+j+m = \mu_1 \\ m \text{ even}}} \frac{m}{2} \; [N_{g_1, |I|+1} (i, \mmu_I)]^{\mathrm{top}} \; [N_{g_2, |J|+1} (j, \mmu_J)]^{\mathrm{top}}\right]^{\mathrm{top}} \label{ntoprecur}
\end{align}
Since both $[N_{g,n}(\mu_1, \ldots, \mu_n)]^{\mathrm{top}}$ and $[Q_{g,n}(\mu_1, \ldots, \mu_n)]^{\mathrm{top}}$ are independent of parities,
we may assume all $\mu_i$ to be even, so that none of $N_{g-1,n+1} (i,j,\mmu_{X\setminus \{1\}})$, $N_{g,n-1} (i,\mmu_{X\setminus \{1,k\}})$, $N_{g_1, |I|+1} (i, \mmu_I)$, $N_{g_2, |J|+1} (j, \mmu_J)$ vanish due to parity issues.

Compare the right hands sides of equations \eqref{toprecur} and \eqref{ntoprecur}. They are identical except for factors of $2$, and that $N_{g,n}$ sums over
even $m$, while $Q_{g,n}$ sums over both even and odd $m$. Proposition \ref{lemma-odd-induction} implies that for $Q_{g,n}$, the top degree component of the
sum over even $m$ in \eqref{toprecur} is the same as that over odd $m$. This introduces another factor of $2$. Comparing the base cases (proposition \ref{prop:Ngn_basecase} for $N_{g,n}$, propositions \ref{basecases} and \ref{basecase2} for $Q_{g,n})$ and recursions on top degree terms (\eqref{ntoprecur} for $N_{g,n}$ and \ref{toprecur} for $Q_{g,n}$), we obtain by induction the desired result.
\end{proof}

We now prove the remaining theorems from the introduction.

\begin{proof}[Proof of theorem \ref{intersection}]
This follows immediately from theorem \ref{Nintersection} and proposition \ref{QN}.
\end{proof}

\begin{proof}[Proof of theorem \ref{Pcount}]
This follows the same argument as proposition \ref{P03count}.
Recall 
$$Q'_{g,n}(\mu_1,\ldots,\mu_n):=\frac{1}{2^{\sum^n_1 \delta_{\mu_i,0}(\mu_1,\ldots,\mu_n)}}Q_{g,n}(\mu_1,\ldots,\mu_n).$$
Since $Q_{g,n}$ is a quasi-polynomial, so is $Q'_{g,n}$. Separating $Q'_{g,n}$ into monomials we see that the right hand side of equation
\eqref{PQ'}
$$P'_{g,n}(\mu_1, \ldots, \mu_n) = \sum_{0 \leq \nu_i \leq \mu_i} \left(Q'(\nu_1, \ldots, \nu_n)\prod_{i=1}^n \binom{2\mu_i}{\mu_i - \nu_i}\right)$$
is a sum of terms of the form 
\begin{align*}
\prod_{i\in X_e}\left(\sum_{\substack{0 \leq \nu_i \leq \mu_i\\ \nu_i \text{ even}}}\nu_i^{2n_i+1}\binom{2\mu_i}{\mu_i - \nu_i}\right)\cdot
\prod_{i\in X_o}\left(\sum_{\substack{0 \leq \nu_i \leq \mu_i\\ \nu_i \text{ odd}}}\nu_i^{2n_i+1}\binom{2\mu_i}{\mu_i - \nu_i}\right)\cdot
\prod_{i\in X_\emptyset}\binom{2\mu_i}{\mu_i}
\end{align*}
where $n_i\leq 3g-3+n$ as the degree of degree of $Q_{g,n}^{(X_e,X_o,X_0)}$ is at most $6g-6+3n-|X_0|$.
By Proposition \ref{almostpoly}, each
$$\sum_{\substack{1 \leq \nu_i \leq \mu_i\\ \nu_i \text{ fixed parity}}}\nu_i^{2n_i+1}\binom{2\mu_i}{\mu_i - \nu_i}$$ 
is of the form
$$\frac{\binom{2\mu_i}{\mu_i}}{(2\mu_i-1)(2\mu_i-3)\dots(2n-2n_i-1)}P_{n_i}(\mu_i)$$ 
for polynomials $P_{n_i}$. Hence taking a common denominator, 
$$P'_{g,n}(\mu_1, \ldots, \mu_n)=\left(\prod_1^n\frac{\binom{2\mu_i}{\mu_i}}{(2\mu_i-1)(2\mu_i-3)\dots(2n-2(3g-3+n)-1)}\right)F_{g,n}(\mu_1,\ldots,\mu_n)$$
for some polynomial $F_{g,n}$. Since $\binom{2\mu_i}{\mu_i} = 2^{\delta_{\mu_i,0}}\binom{2\mu_i-1}{\mu_i}$, $P_{g,n}$ has the
required form. 
\end{proof}

A nice way to express the relationship \eqref{PQ'} is
to package $P_{g,n}$ and $Q_{g,n}$ into generating differentials.
For $g \geq 0$ and $n \geq 1$ let
\begin{align*}
\omega^P_{g,n}(x_1, \ldots, x_n) &=
\sum_{\mu_1, \ldots, \mu_n \geq 0} P'_{g,n}(\mu_1, \ldots, \mu_n) x_1^{-\mu_1 - 1} \cdots x_n^{-\mu_n - 1} \; dx_1\cdots dx_n \\ 
\omega^Q_{g,n}(z_1, \ldots, z_n) &= 
\sum_{\nu_1, \ldots, \nu_n \geq 0} Q'_{g,n}(\nu_1, \ldots, \nu_n) z_1^{\nu_1 - 1} \cdots z_n^{\nu_n - 1}\; dz_1 \cdots dz_n. 
\end{align*}

Following 
\cite{DKM2017} and
\cite{DN2013}, 
for any quasi-polynomial $f$,
\[
\omega^f(z_1, \ldots, z_n) = 
\sum_{\nu_1, \ldots, \nu_n \geq 0} f(\nu_1, \ldots, \nu_n) z_1^{\nu_1 - 1} \cdots z_n^{\nu_n - 1}\; dz_1 \cdots dz_n
\]
is a meromorphic differential, hence
$\omega^Q_{g,n}$ is a meromorphic differential.
Using techniques from that previous work, one can show the following.
\begin{proposition}
$\omega^Q_{g,n}$ is the pullback of $\omega^P_{g,n}$ under the map $x_i = \frac{(1+z_i)^2}{z_i}$.
\qed
\end{proposition}

\appendix

\section{Proofs of combinatorial identities}

We now give elementary proofs of the statements from section 
\ref{sec:preliminaries}

Recall proposition \ref{almostpoly} states that there are polynomials $P_\alpha, Q_\alpha$ such that 
\begin{align*}
\sum_{0\leq i\leq n \text{ even}}{i^{2\alpha+1}\binom{2n}{n-i}} &= \frac{\binom{2n}{n}}{(2n-1)(2n-3)\dots(2n-2\alpha-1)}P_{\alpha}(n) \\
\sum_{0\leq i\leq n \text{ odd}}{i^{2\alpha+1}\binom{2n}{n-i}} &= \frac{\binom{2n}{n}}{(2n-1)(2n-3)\dots(2n-2\alpha-1)}Q_{\alpha}(n)  
\end{align*}

\begin{proof}[Proof of proposition \ref{almostpoly}]
For $\alpha=0$, we have 
\begin{align*}
i\binom{2n}{n-i} &= \frac{(2n-1)[(2n-1)-(2n-2i-1)]}{2(2n-1)}\binom{2n}{n-i} \\
&= \frac{[((2n-1)-(n-i-1))((2n-1)-(n-i))-(n-i)(n-i-1)]}{2(2n-1)}\binom{2n}{n-i} \\
&= \frac{1}{2(2n-1)}\left[(n-i+2)(n-i+1)\binom{2n}{n-i+2} - (n-i)(n-i-1)\binom{2n}{n-i}\right] 
\end{align*}
Therefore both sums telescope and 
\begin{align*}
\sum_{0\leq i\leq n \text{ even}}{i\binom{2n}{n-i}} &= \frac{1}{2(2n-1)}(n+2)(n+1)\binom{2n}{n+2} = \frac{n(n-1)}{2(2n-1)}\binom{2n}{n}\\
\sum_{0\leq i\leq n \text{ odd}}{i\binom{2n}{n-i}} &= \frac{1}{2(2n-1)}(n+1)(n)\binom{2n}{n+1} = \frac{n^2}{2(2n-1)}\binom{2n}{n} 
\end{align*}
It follows that $P_0(n)=\frac{n^2-n}{2}$, $Q_0(n)=\frac{n^2}{2}$.
For $\alpha>0$, we have 
\begin{align*}
i^{2\alpha+3}\binom{2n}{n-i} &= n^2 i^{2\alpha+1}\binom{2n}{n-i} - (n+i)(n-i)i^{2\alpha+1}\binom{2n}{n-i}\\
&= n^2 i^{2\alpha+1}\binom{2n}{n-i} - 2n(2n-1)i^{2\alpha+1}\binom{2n-2}{(n-1)-i}
\end{align*}
By induction
\begin{align*}
\sum_{0\leq i\leq n \text{ even}}{i^{2\alpha+3}\binom{2n}{n-i}} = &n^2\frac{\binom{2n}{n}}{(2n-1)(2n-3)\dots(2n-2\alpha-1)}P_{\alpha}(n) \\ 
&- 2n(2n-1)\frac{\binom{2n-2}{n-1}}{(2n-3)\dots(2n-2\alpha-3)}P_{\alpha}(n-1)
\end{align*}
It follows that 
\begin{equation}
\label{recursion_for_P}
P_{\alpha+1}(n)=n^2[(2n-2\alpha-3)P_{\alpha}(n)-(2n-1)P_{\alpha}(n-1)]
\end{equation}
and similarly
\begin{equation}
\label{recursion_for_Q}
Q_{\alpha+1}(n)=n^2[(2n-2\alpha-3)Q_{\alpha}(n)-(2n-1)Q_{\alpha}(n-1)]
\end{equation}
are polynomials in $n$.
\end{proof}

Using $P_0, Q_0$ calculated above and the recursions \eqref{recursion_for_P} and \eqref{recursion_for_Q}, we immediately obtain the identities of equations \eqref{PQ01}--\eqref{comb_id_3o}.

Recall proposition \ref{lemma-odd} states that for positive odd $k_1, k_2$ and fixed parities of $i_1, i_2$, the sum of $i_1^{k_1} i_2^{k_2}$ over $i_1, i_2 \geq 1$ such that $i_1 + i_2$ is an odd polynomial of degree $k_1 + k_2 + 1$, with leading coefficient independent of choice of parities.

\begin{proof}[Proof of proposition \ref{lemma-odd}]
Let $S_k(n)$, $S_k^{\text{e}}(n)$, $S_k^{\text{o}}$ be the $k$-th power sum, the even and odd $k$-th power sums:
\[
S_k(n)=\sum_{1\leq i\leq n} i^k, \quad \quad
S_k^{\text{e}}(n)=\sum_{1\leq i\leq n, \ i \text{ even}} i^k, \quad \quad
S_k^{\text{o}}(n)=\sum_{1\leq i\leq n, \ i \text{ odd}} i^k.
\]
Let $B_{i}$ the $i$-th Bernoulli number. A well known argument gives Faulhaber's formula
\begin{equation}\label{eqF1}
S_k(n) = \frac{1}{k+1}\sum_{0\leq i\leq k}(-1)^i\binom{k+1}{i}B_i n^{k+1-i} = n^k + \frac{1}{k+1}\sum_{0\leq i\leq k}\binom{k+1}{i}B_i n^{k+1-i},
\end{equation}
and a similar generating functions argument shows that 
\begin{align}
S_k^{\text{e}}(n)&= n^k + \frac{1}{2(k+1)}\sum_{0\leq i\leq k}2^{i}\binom{k+1}{i}B_i n^{k+1-i}, & \text{ if } n \text{ is even,} \label{eqF2}\\
S_k^{\text{o}}(n)&= n^k + \frac{1}{2(k+1)}\sum_{0\leq i\leq k}2^{i}\binom{k+1}{i}B_i (n^{k+1-i}-1), & \text{ if } n \text{ is odd}\label{eqF3}.
\end{align}
Since the odd Bernoulli numbers are zero except $B_1 = -\frac{1}{2}$, Equations \eqref{eqF1}, \eqref{eqF2}, and \eqref{eqF3} imply that $S_k^{\text{e}}(n)$ and $S_k^{\text{o}}(n)$ are even or odd polynomials depending on the parity of $(k+1)$, with the possible exception of the constant term and $n^k$ term. The coefficient of $n^k$ in $S_k^{\text{e}}(n)$ is $\frac{1}{2}$ if $n$ is even and $0$ otherwise. The coefficient of $n^k$ in $S_k^{\text{o}}(n)$ is $\frac{1}{2}$ if $n$ is odd and $0$ otherwise. If $n$ is even, then the constant terms in $S_k^{\text{e}}(n)$ and $S_k^{\text{o}}(n)$ are both $0$. If $n$ is odd,  then the constant terms in $S_k^{\text{e}}(n)$ and $S_k^{\text{o}}(n)$ are
$$\pm C_k:=\pm\frac{1}{2(k+1)}\sum_{0\leq i\leq k}2^{i}\binom{k+1}{i}B_i.$$ Observe that $C_k/k!$ is the 
coefficient of $x^k$ in $\left(\frac{e^x-1}{2x}\right)\left(\frac{2x}{e^{2x}-1}\right)=\frac{1}{e^x+1}$. Since $\frac{1}{e^{2x}+1}+\frac{1}{e^{-2x}+1}=1$, $C_k = 0$ for positive even $k$.

If the fixed parity of $i_1$ is odd, then  
\[
\sum_{\substack{i_1,i_2\geq 1, \ \ i_1+i_2=n \\ \{i_1,i_2\} \text{ have fixed parities}}}i_1^{k_1}i_2^{k_2}
= \sum_{\substack{1\leq i_1\leq n,\\ i_1\text{ odd}}}i_1^{k_1}(n-i_1)^{k_2} 
=\sum_{0\leq j\leq k_2} (-1)^{k_2-j} \binom{k_2}{j} n^j S_{k_1+k_2-j}^{\text{o}}(n).
\]
Since $(k_1+k_2+1)$ is odd, each term $(-1)^{k_2-j} \binom{k_2}{j} n^j S_{k_1+k_2-j}^{\text{o}}(n)$ is almost an odd polynomial except for 
the constant and $n^{k_1+k_2-j}$ term in $S_{k_1+k_2-j}^{\text{o}}(n)$. The coefficient of $n^{k_1+k_2-j}$ is $\frac{1}{2}$ if
$n$ is odd and $0$ is $n$ is even. Hence the overall contribution to $\sum (-1)^{k_2-j} \binom{k_2}{j} n^j S_{k_1+k_2-j}^{\text{o}}(n)$ is $0$ in both cases, as $\sum_{0\leq j\leq k_2}(-1)^j \binom{k_2}{j} = 0$. The constant term in $S_{k_1+k_2-j}^{\text{o}}(n)$ is $0$ unless $(k_1+k_2-j)$ is odd, i.e., $j$ is odd, so it contributes an odd degree term $(-1)^{k_2-j}\binom{k_2}{j}C_{k_1+k_2-j}n^j$ to $\sum (-1)^{k_2-j} \binom{k_2}{j} n^j S_{k_1+k_2-j}^{\text{o}}(n)$. Therefore overall $\sum (-1)^{k_2-j} \binom{k_2}{j} n^j S_{k_1+k_2-j}^{\text{o}}(n)$ is an odd polynomial of $n$.

Similarly if $i_1$ is even, then 
\[
\sum_{\substack{i_1,i_2\geq 1, \ \ i_1+i_2=n \\ \{i_1,i_2\} \text{ have fixed parities}}}i_1^{k_1}i_2^{k_2}
= \sum_{\substack{1\leq i_1\leq n,\\ i_1\text{ even}}}i_1^{k_1}(n-i_1)^{k_2} 
=\sum_{0\leq j\leq k_2} (-1)^{k_2-j} \binom{k_2}{j} n^j S_{k_1+k_2-j}^{\text{e}}(n)
\]
is also an odd polynomial of $n$.

Finally, it follows easily from induction that 
$$\sum_{0\leq i\leq n}\frac{(-1)^i}{x+i}\binom{n}{i} = \frac{n!}{x(x+1)\cdots(x+n)}.$$
Hence by Equations \eqref{eqF1}, \eqref{eqF2} and \eqref{eqF3}, the leading coefficient of $\sum_{\substack{i_1,i_2\geq 1, \ \ i_1+i_2=n \\ \{i_1,i_2\} 
\text{ have fixed parities}}}i_1^{k_1}i_2^{k_2}$, regardless of the choice of parities, is, 
$$\sum_{0\leq j\leq k_2} \frac{(-1)^{k_2-j}}{2(k_1+k_2+1-j)}\binom{k_2}{j} = \left(2(k_2+1)\binom{k_1+k_2+1}{k_2+1}\right)^{-1}> 0.$$ 
Therefore the odd polynomial has degree $(k_1+k_2+1)$, and the leading coefficient is independent of the choice of parities.
\end{proof}

Lemma \ref{lem-odd-even-power-sums} simply gives explicit expressions for $S^o_1 (n), S^o_2 (n), S^e_1 (n)$ and $S^e_2 (n)$, which follow immediately from \eqref{eqF1} and \eqref{eqF2}.

\end{document}